\documentclass{article}

\usepackage{arxiv}

\usepackage[T1]{fontenc}    
\usepackage{hyperref}       
\usepackage{url}            
\usepackage{booktabs}       
\usepackage{amsfonts}       
\usepackage{nicefrac}       
\usepackage{microtype}      
\usepackage{amsmath}
\usepackage{amssymb}
\usepackage{amsmath, amsthm}
\usepackage{graphicx}
\usepackage[colorinlistoftodos]{todonotes}
\usepackage{datetime}
\theoremstyle{definition}
\newtheorem{definition}{Definition}[section]
\newtheorem{theorem}{Theorem}[section]
\newtheorem{lemma}[theorem]{Lemma}
\newtheorem{proposition}[theorem]{Proposition}
\newtheorem{corollary}[theorem]{Corollary}
\newtheorem{example}[theorem]{Example}
\newtheorem{remark}[theorem]{Remark}

\newcommand\norm[1]{\left\lVert#1\right\rVert}
\newcommand{\turnnw}[1]{\rotatebox[origin=c]{270}{\ensuremath#1}}

\title{Characterization of the existence of 
semicontinuous Richter-Peleg multi-utility representations}

\author{
  Athanasios Andrikopoulos\thanks{Associate professor  (https://www.ceid.upatras.gr/webpages/faculty/aandriko/)} \\
  Dept. of Computer Engineering and Informatics\\
  University of Patras\\
  Patras, 26504, Greece \\
  \texttt{aandriko@ceid.upatras.gr} \\
}

\begin{document}
\maketitle

\begin{abstract}
The utility representations of preference relations
in symmetric topological spaces have the advantage of fully characterizing these relations.
But this is not true in the case of representations of preference relations that are mostly incomplete and use asymmetric topologies.
To avoid this unfortunate circumstance, due to lack of symmetry,
the notion of semicontinuous Richter-Peleg multi-utility representation was first introduced and studied by Minguzzi \cite{min}.
Generally, the study of semicontinuous functions reveals that topology and order are two aspects of the same mathematical object
and therefore should be studied jointly. 
The mathematical object that can be used to model and analyze the concepts of asymmetry and duality that arise from the coexistence of topology and order is that of bitopological preordered space.
In this paper, we characterize the existence of 
semicontinuous Richter-Peleg multi-utility representations in bitopological spaces.
Based on this characterization, we prove that 
a preorder $\precsim$ has the set of all Scott and $\omega$-continuous functions as a
Richter-Peleg multi-utility representation if and only if $\precsim$ is precontinuous
in the sense of Ern\'{e} \cite{ern2}.
\end{abstract}


\section{Introduction} 
 Traditionally, a preference relation on a nonempty choice set $X$ 
is  a reflexive,
 transitive and complete binary relation $\precsim$, and a utility function $u: X\to \mathbb{R}$ {\it represents} $\precsim$
if it satisfies the
property that $x\precsim y$ if and only if $u(x)\precsim u(y)$.
By utilizing topology, utility representations make it easier to find the best choices.
With such a utility representation, preference relations are complete.
Two classical results on representations of binary relations are those of Eilenberg \cite{eil} and Debreu \cite{deb1}, \cite{deb2}. The Eilenberg utility representation theorem states that every continuous complete preorder on a separable connected topological space is representable by a utility function.
The Debreu utility representation theorem states that every continuous complete preorder on a second countable topological space has
a continuous utility representation.
In 1971 Arrow and Hahn \cite[Theorem 1]{AH}
introduced a method for representing a preference relation by a continuous utility function
on a subset of Euclidean space that satisfies certain conditions.
Since then, many representation theorems 
of binary relations have been established. 
Rader \cite{rad} and Debreu \cite{deb2} considered the existence of semicontinuous utilities for complete and transitive orderings, Jaffray \cite{Jaf} considered partial orderings, and Peris and Subiza \cite{PS} considered acyclic binary relations.
Richter provides alternate proof of Rander's result.
As stated previously, such a utility representation provides us complete information about the binary relation while also allowing us to work with real functions.
However, if the preference relation $\precsim$ is not complete, such a representation does not exist and only weak representations can be found.
That is, there is a real
utility function $u$ in which
$x\sim y$ implies $f(x)=f(y)$ and
$x\prec y$ implies $u(x)<u(y)$ \cite[Page 146]{min}.
This notion of utility representation, which was first used by Aumann \cite[Page 447]{aum}
in the cardinal framework of preferences over objective lotteries,
does not provide complete information about the preference relation. 
Indeed, if we observe
that $u(x)\geq u(y)$, we only know that it is not the case that $x\prec y$, but we do not
know whether $y\prec x$ holds or not. 
Nonetheless, it is useful because every element that maximizes this function is a maximal element, and thus the existence of maxima of this function ensures the existence of maximal elements.
For separable and spacious strict partial orders, Peleg \cite{pel} obtains continuous and semicontinuous weak utility functions. 
Peris and Subiza \cite{PS} also provide new sufficient conditions for
the existence of semicontinuous weak utilities for acyclic binary relations
and
Alcantud and Rodriguez-Palmero \cite{AR} follow a
different approach and provide necessary and sufficient conditions for the existence
of semicontinuous weak utilities.
Since Peleg \cite{pel} and Richter \cite{ric} were the first to look into the idea of weak representation,
is often referred to as the {\it Richter-Peleg utility representation}.
Another approach for representing incomplete binary relations is to find a collection $\mathcal{V}$ of (continuous) increasing functions on a decision space $X$ 
that characterizes the preference relation 
in the sense that $x\precsim y$ for any $x, y\in X$
if and only if $u(x)\precsim u(y)$ for all $u\in \mathcal{V}$.
The advantage of this approach, known as the {\it {\rm (}continuous{\rm )} multi-utility representation} of $\precsim$, is that it fully characterizes the preference relation (see \cite{BH1}, \cite{EO}).
As demonstrated by Bewley \cite{bew} and Dubra et al. \cite{DMO},
this kind of representation arises intrinsically from the classical frameworks of Anscombe-Aumann, Savage, and Von Neumann-Morgenstern, respectively, upon relaxation of the completeness axiom.
A {\it Richter-Peleg multi-utility representation} $\mathcal{V}$ for an incomplete binary relation 
$\precsim$ on $X$ is a multi-utility representation for $\precsim$ in which
every function
$v\in \mathcal{V}$
is a Richter-Peleg utility for $\precsim$, i.e., 
$x\precsim y$ implies that $v(x)\precsim v(y)$ and $x\prec y$ implies that $v(x)<v(y)$.
If $\precsim$ is complete, then any Richter-Peleg representation $\mathcal{V}$ of $\precsim$ is the standard utility representation, that is, $x\precsim y$ for each
$x, y\in X$ if and only if $u(x)\precsim u(y)$.
The theory of the (continuous) multi-utility representation in its modern form was first investigated by 
Ok \cite{ok}, Mandler \cite{man} and 
Evren and OK \cite{EO}.
Continuing the work of Evren and Ok \cite{EO},
Bosi and Herden \cite[Theorem 3.4]{BH1}
show that if $\precsim$ is a weakly continuous\footnote{A preorder $\precsim$ on a topological space $(X,\tau)$ is said to be {\it weakly continuous}  (see Bosi and Herden, \cite{BH3}, \cite{BH4})
if for every $x, y\in X$ with $x\prec y$ there exists a continuous increasing real-valued function $f_{xy}$ on $(X,\tau)$ such that $f_{xy}(x) < f_{xy}(y)$.} preorder in a second countable topological space $(X,\tau)$, then $\precsim$ has a continuous multi-utility representation if and only if it satisfies the continuous analogue of the Dushnik-Miller Theorem.
The notion of a Richter-Peleg multi-utility representation,
which refers to a multi-utility representation that consists of strictly increasing functions,
was introduced by Minguzzi \cite[Sec. 5]{min}.
Recently,
Alcantud, Bosi and Zuanon 
\cite{ABZ} as well as
Bosi, Estevan and Ravent\'{o}s-Pujol
\cite{BER}, investigated the existence of (semicontinuous) Richter-Peleg multi-utilities.
The main result in \cite{ABZ} basically states that:
A preorder $\precsim$ on (a topological space $(X,\tau)$) a set $X$
is represented by a (upper/lower semicontinuous) Richter-Peleg multi-utility
if and only if admits a (upper/lower semicontinuous)
Richter-Peleg representation. On the other hand, in \cite{BER} the authors
studied
necessary and sufficient topologies for the existence of a semicontinuous and finite 
Richter-Peleg multi-utility representation for a preorder.

In this paper, we use a bitopological approach  to characterize the existence of semicontinuous Richter-Peleg multi-utility representations. In fact, we use joincompact spaces, which are the bitopological analogues of compact Hausdorff spaces and naturally occur as Lawson-closed subsets of continuous lattices with the Scott and lower topologies restrictions.
In light of this characterization, we demonstrate that a preorder 
$\precsim$ has the set of all Scott and $\omega$-continuous functions as
Richter-Peleg multi-utility representation if and only if $\precsim$ is precontinuous
in the sense of Ern\'{e} \cite{ern2}.

\section{Notations and definitions}

Let $X$ be a non-empty decision space and
$\precsim\subseteq X\times X$ be a binary relation on $X$. As usual, $\prec$ denotes the {\it strict part} of $\precsim$.
We sometimes
abbreviate $(x,y)\in \precsim$ (resp. $(x,y)\in \prec$) as $x\precsim y$ (resp. $x\prec y$).
We say that $\precsim$ on $X$ is (i)
{\it reflexive} if for each $x\in X$, $x\precsim x$; (ii)
{\it asymmetric} if for all $x, y\in X$, $x\precsim y$
 and $y \not\precsim x$;
(iii) {\it transitive} if for all $x,y,z\in X$, [$x\precsim z$ 
and
$z\precsim y$ $\Rightarrow x\precsim y$; 
(iv)
{\it antisymmetric} if for each $x,y\in X$,
[$x\precsim y$ and
$y\precsim x] \Rightarrow x=y$; 
(vi) {\it complete} if
for all $x$ and $y$, $x\precsim y$ or $y\precsim x$; 
(v) {\it total} if for each $x,y\in X$,
$x\neq y$ we have $x\precsim y$ or $y\precsim x$.
The following combination of properties are considered in the next
theorems. A binary relation $\precsim$ on $X$ is:  (1) {\it A preorder} if
$\precsim$ is reflexive and transitive;
(2) {\it A partial order} if
$\precsim $ is reflexive, transitive and antisymmetric;
(3) {\it A linear order} if
$\precsim$ is a total partial order.
A {\it preordered set} is a pair $(X,\precsim)$ consisting of a set $X$ and a preorder $\precsim$ on $X$.
Let $\mathcal{L}=(X,\precsim)$ be a preordered set.
If $\precsim$ is an order, then $\mathcal{L}$ is called {\it partially ordered set} or {\it poset}.
A {\it lattice} is a mathematical structure studied in the mathematical subdisciplines of order theory and abstract algebra, among others.
It consists of a poset $(X,\precsim)$ in which every pair of elements $x$ and $y$
has a unique supremum $x\vee y$ (also called a least upper bound or join) and a unique 
infimum  $x\wedge y$ (also called a greatest lower bound or meet).
A subspace $(A,\precsim)$ of a lattice $(X,\precsim)$ is called a {\it sublattice} of $(X,\precsim)$ if it becomes a 
lattice with respect to $\precsim$.
If $\precsim$ is a preorder on $X$, then we denote the associated asymmetric relation $\prec$
and the associated equivalence relation $\sim$, respectively, by 
[$x\prec y\Leftrightarrow (x\precsim y)\wedge (y\not\precsim x)]$ and 
[$x\sim y\Leftrightarrow (x\precsim y)\wedge (y\precsim x)]$. 
We recall that $f: (X,\precsim)\to (\mathbb{R},\leq)$
is: (i) {\it increasing} if, for each $x,y\in X$, $x\precsim y$ implies $f(x)\leq f(y)$;
(ii) {\it order preserving} if for all $x ,y\in X$,
 $x\prec y\Rightarrow f(x)<f(y)$].
A {\it Richter-Peleg utility} is an increasing function $f: (X,\precsim)\to (\mathbb{R},\leq)$ 
that is also an order-preserving function
(see e.g. Peleg \cite{pel} and Richter \cite{ric0}). 
Equivalently, if $\sim$ implies $f(x)=f(y)$ and it is order preserving.
We say that a preorder $\precsim$ admits
a {\it Richter-Peleg multi-utility representation} by a family of functions $\mathcal{V}$ if
$x\precsim y \Leftrightarrow  f(x)\leq f(y)$ for all $f\in\mathcal{V}$. 
It is obvious that a Richter-Peleg multi-utility representation $\mathcal{V}$ of a preordered set
$(X,\precsim)$
characterizes the strict part $\prec$ of $\precsim$, in the sense that for each $x,y \in X$,
$x\prec y \Leftrightarrow  f(x)<f(y)$ for all $f\in\mathcal{V}$.

\par
Let $(X,\tau)$ be a topological space. 
We say that $(X,\tau)$
is a: ($\mathfrak{i}$) $T_{_0}$ topological space if given two distinct points in it, there exists an open neighborbood of it that contains exactly one of them;
($\mathfrak{ii}$) $T_{_2}$ or {\it Hausdorff} topological space
if its distinct points are contained in disjoint neighborhoods and ($\mathfrak{iii}$)
{\it quasi-compact} if for each collection of open sets which covers $X$ there exists a finite subcollection that also covers $X$.
If it is quasi-compact and Hausdorff, it is called {\it compact}. 
A function $f$ in $(X,\tau)$ is {\it upper} (resp., {\it lower})  
 {\it semicontinuous}
at $x\in X$ if for each $\varepsilon>0$,
there exists a neighborhood $U_{_x}$ of $x$ 
such that for all $y\in U_{_x}$ we have $f(y)<f(x)+\varepsilon$ (resp. $f(y)>f(x)-\varepsilon$).
An {\it isomorphism} is a structure-preserving mapping $f$ that can be reversed by $f^{-1}$ between two structures of the same type.
If there is an isomorphism between two mathematical structures, they are said to be {\it isomorphic}.
If $A\subseteq X$, then is the {\it interior} (resp. {\it closure}) of $A$ with respect to $\tau$ topology is denoted by 
$int_{_{\tau}}A$ (resp. $cl_{_{\tau}}A$).
Every topology $\tau$ on a set $X$ induces a preorder $\precsim_{_\tau}$ on this set, called {\it specialization preorder}, as follows:
$x\precsim_{_\tau} y$ if and only if $x\in cl_{_\tau}y$, $x, y\in X$.
This preorder is an order if and only if $(X,\tau)$ is $T_{_0}$. The {\it compatible topologies}
on a preordered set are those which induce the given preorder.
Nachbin \cite{nac} defines a {\it topological preordered space} $(X,\tau,\precsim)$ as a topological space $(X,\tau)$ with a preorder $\precsim$ that is closed as a subset of $X\times X$.
If $(X,\tau)$ is Hausdorff, then $\precsim$ is a partial order and $(X,\tau,\precsim)$ 
is called a {\it topological ordered space}.
For any $A\subseteq X$, define $\displaystyle\uparrow A=\{y\in X\vert x\precsim y\ {\rm for\ some}\ x\in A\}$.
We also write $\displaystyle\uparrow \{x\}$ as $\displaystyle\uparrow x$.
The sets $\displaystyle\downarrow A$ as $\displaystyle\downarrow x$
are defined dually.
A subset $A$ of $X$ is said to be an increasing (resp. decreasing) set, or to be an upper (lower) set,
if
$A=\displaystyle\uparrow A$ (resp. $A=\displaystyle\downarrow A$).
In what follows, we use two names for the sets in the preceding definition: "increasing (resp. decreasing) set" as in general topology, or "upper (lower) set" as in order theory.
A subset $A$ of $X$ is increasing if and only if $X\setminus A$ is decreasing. For each subset $A$ of $X$ 
there is a smallest increasing set $i(A)$ (decreasing set $d(A)$) that contains $A$.
If $A=\{a\}$ for some $a\in X$, then the notation is $i(a)$ (resp. $d(a)$).
A set $X$ equipped with two topologies $\tau_{_1}$ and $\tau_{_2}$
is called a {\it bitopological space}.
The concept of bitopological spaces was first introduced by Kelly \cite{kel}
in order to generalize the notion of topological space.
Every bitopological space $\mathfrak{D}=(X,\tau_{_1},\tau_{_2})$
can be regarded as a topological space 
$(X,\tau)$ with $\tau_{_1}=\tau_{_2}=\tau$. 
The dual of $\mathfrak{D}$ is the bitopological space
$\mathfrak{D}^{\ast}=(X,\tau_{_2},\tau_{_1})$.
A bitopological space $\mathfrak{D}=(X,\tau_{_1},\tau_{_2})$ is called:
(i) {\it  Pairwise Hausdorff} if for each two points $x$ and $y$ in $X$, there exists$\tau_{_i}$-open neighborhood of $x$
and $\tau_{_j}$-open neighborhood $V_j$ of $y$
such that $U_i\cap V_j=\emptyset$, \  $i, j\in \{1,2\}, i\neq j$.
(ii) {\it Pairwise normal} if for every pair of a $\tau_{_1}$-closed set $F$ and a $\tau_{_2}$-closed set $G$ with 
$F\cap G=\emptyset$, there exist a $\tau_{_1}$-open set $U$ and a $\tau_{_2}$-open set $V$
such that $F\subset V$, $G\subset U$ and $V\cap U=\emptyset$.
A bitopological ordered space $\mathfrak{X}=(X,\tau_{_1},\tau_{_2},\precsim)$ is called
{\it Pairwise normally ordered} if for every pair of a decreasing $\tau_{_1}$-closed set $F$ and an increasing $\tau_{_2}$-closed set $G$ with 
$F\cap G=\emptyset$, there exist an increasing $\tau_{_1}$-open set $U$ and a decreasing $\tau_{_2}$-open set $V$
such that $F\subset V$, $G\subset U$ and $V\cap U=\emptyset$.

\section{Bitopological preordered spaces}

A {\it bitopological preordered space} 
$\mathfrak{X}=(X,\tau_{_1},\tau_{_2},\precsim)$ (see \cite[Definition 1.4]{and}) is a bitopological space 
$\mathfrak{D}=(X,\tau_{_1},\tau_{_2})$
equipped with a $\tau_{_1}\times \tau_{_2}$-closed preorder $\precsim$ in $X\times X$. 
The definition of $\mathfrak{X}$ allows for the definition of the dual bitopological preordered space $\mathfrak{X}^{\ast}=(X,\tau_{_2},\tau_{_1},\precsim^{\ast})$
of $\mathfrak{X}$, where $\precsim^{\ast}=(\precsim)^{-1}$.
If $\mathfrak{X}$ is pairwise Hausdorff, then
$(X,\tau_{_1}\vee\tau_{_2})$
is Hausdorff and $\precsim$ is a partial order ($\precsim$ is antisymmetric). In this case, $\mathfrak{X}$ is called a {\it bitopological ordered space}.
This definition extends the notion of topological ordered space\footnote{Acording to Nachbin, a {\it topological preordered space} $(X,\tau,\precsim)$ is a topological space $(X,\tau)$ equipped with a 
preorder $\precsim$ which is $\tau\times \tau$-closed subset of $X\times X$.
If $(X,\tau)$ is Hausdorff, then $\precsim$ is a partial order and $(X,\tau,\precsim)$ is called a {\it topological ordered space}.}
of Nachbin to the bitopological spaces.
Every bitopological ordered space $\mathfrak{X}$ can be thought of as a
bitopological space $\mathfrak{D}=(X,\tau_{_1},\tau_{_2})$, where $\precsim$ is the equality relation 
$\Delta=\{(x,x)\vert x\in X\}$.
Each member of $\tau_{_1}$ (resp. $\tau_{_2}$) in 
$\mathfrak{X}$
is called an {\it open}
set, or a $\tau_{_1}$-{\it open} (resp. $\tau_{_2}$-{\it open}) set if one wishes to emphasize the topologies $\tau_{_1}$ (resp. $\tau_{_2}$) on $X$. 
This notion comes in handy when addressing many topologies on a given $X$ at the same time.

\begin{definition}\label{osr}{\rm Let $\mathfrak{X}=(X,\tau_{_1},\tau_{_2},\precsim)$ 
be a bitopological preordered space. A subset $N$ of $X$ is a 
$\tau_{_1}$-neighborhood (resp. $\tau_{_2}$-neighborhood) of an $x\in X$
if and only if $N$ includes a $\tau_{_1}$-open (resp. $\tau_{_2}$-open) set
containing $x$. If $N$ is a 
$\tau_{_1}$-open (resp. $\tau_{_2}$-open) set then it is called
$\tau_{_1}$-open neighborhood (resp. $\tau_{_2}$-open neighborhood) of $x$.
}
\end{definition}

\begin{example}\label{ex0}{\rm Let $\mathfrak{I}=(\mathbb{R},\mathfrak{U},\mathfrak{L},\precsim)$,
where $\mathfrak{U}$, $\mathfrak{L}$ are the upper and lower topologies:
$\{(a,\infty)\vert a\in \mathbb{R}\}$, $\{ (-\infty,a)\vert \ a\in \mathbb{R}\}$ and 
$\precsim$ is the usual order in $\mathbb{R}$. Then it is easy to check that $\mathfrak{I}$ 
is a bitopological preordered space.
}
\end{example}

In what follows,
$\mathfrak{I}$ will denote the bitopological preordered space 
$(\mathbb{R},\mathfrak{U},\mathfrak{L},\precsim)$.

\begin{example}\label{ex1}{\rm Let $(X,\tau)$ be a topological space, $\tau^{\ast}$ be the discrete topology on $X$ and $\precsim_{_{\tau}}$ be the specialization order of $\tau$. Then,
$(X,\tau,\tau^{\ast},\precsim_{_{\tau}})$ is a bitopological preordered space, that is,  $\precsim_{_{\tau}}$ is 
$\tau\times \tau^{\ast}$-closed subset of $X\times X$. Indeed,
if 
$(a,b)\in X\times X\setminus \{(x,y)\in X\times X\vert \ x\precsim_{_{\tau}}y\}$, then there exists a $\tau$-open
neighborhood $V_{a}$ of $a$ such that 
$V_{a}\cap \{b\}=\emptyset$ or equivalently
$b\notin V_{_a}$. Then,
$V_{_a}\times \{b\}\subset X\times X\setminus \{(x,y)\in X\times X\vert \ x\precsim_{_{\tau}}y\}$.
Indeed, suppose to the contrary that there exists $z\in X$ such that
$(z,b)\in V_{_a}\times \{b\}$ and $(z,b)\notin X\times X\setminus \{(x,y)\in X\times X\vert \ x\precsim_{_{\tau}}y\}$. Then, $z\precsim_{_{\tau}}b$ which implies that $b\in \bigcap V_{z}$ ($V_{z}$ denotes
an arbitrary $\tau$-open neighborhood of $z$). Since $z\in V_a$ and $V_a$ is 
$\tau$-open, there exists a $\tau$-open neighborhood $V^{\ast}_{z}$ of $z$ such that
$V^{\ast}_{_z}\subset V_a$. But then, $b\in \bigcap V_{z}\subseteq V^{\ast}_{_z}\subset V_a$, a contradiction.
In light of this last contradiction, it can be concluded that
$\precsim_{_{\tau}}$ is $\tau\times \tau^{\ast}$-closed.
}
\end{example}

\begin{example}\label{ex2}{\rm Let $(X,d)$ be a quasi-pseudometric space\footnote{A quasi-pseudometric space $(X,d)$ is a set $X$ together with a non-negative real-valued function $d: X\times X\longrightarrow \mathbb{R}$ (called a quasi-pseudometric) such that, for every $x, y, z\in X$: (i) $d(x,x)=0$;
(ii) $d(x,y)\leq d(x,z)+d(z,y)$. A quasi-pseudometric $d$ on $X$ induces a topology $\tau_{d}$ on $X$ 
which has as a base the family of $d$-balls $\{ B_{_d}(x,r): x\in X, r>0\}$ where $B_{_d}(x,r)=\{y\in X: d(x,y)<r\}$.}.
Then, $(X,\tau_{_d},\tau_{_{d^{-1}}},\precsim_{_{\tau_{_d}}})$ is a bitopological preordered space.
To demonstrate this, we must show that $cl_{_{\tau_{_d}\times \tau_{_{d^{-1}}}}}G_{_{\tau_{_d}}}=G_{_{\tau_{_d}}}$,
where $G_{_{\tau_{_d}}}$ is the graph of $\precsim_{_{\tau_{_d}}}$ and $cl_{_{\tau_{_d}\times \tau_{_{d^{-1}}}}}G_{_{\tau_{_d}}}$
is the closure of $G_{_{\tau_{_d}}}$ in $(X\times X, \tau_{_d}\times \tau_{_{d^{-1}}})$. 
Indeed, let $(a,b)\in cl_{_{\tau_{_d}\times \tau_{_{d^{-1}}}}}G_{_{\tau_{_d}}}$. Then, for each $\varepsilon>0$, we have 
$B_{_d}(a,\frac{\varepsilon}{3})\times B_{_{d^{-1}}}(b,\frac{\varepsilon}{3})\cap G_{_{\tau_{_d}}}\neq \emptyset$. Therefore, there exists 
$(z_{_1},z_{_2})\in B_{_d}(a,\frac{\varepsilon}{3})\times B_{_{d^{-1}}}(b,\frac{\varepsilon}{3})$ and 
$(z_{_1},z_{_2})\in G_{_{\tau_{_d}}}$ or equivalently
$d(a,z_{_1})<\frac{\varepsilon}{3}$, $d(z_{_2},b)<\frac{\varepsilon}{3}$ and $d(z_{_1},z_{_2})=0<\frac{\varepsilon}{3}$. We have
$d(a,b)\leq d(a,z_{_1})+d(z_{_1},z_{_2})+d(z_{_2},b)<\frac{\varepsilon}{3}+\frac{\varepsilon}{3}+\frac{\varepsilon}{3}=\varepsilon$.
As a result, we can conclude that $d(a,b)<\varepsilon$ is true for every $\varepsilon>0$.
It follows that
$d(a,b)=0$ which implies that $(a,b)\in G_{_{\tau_{_d}}}$.

}
\end{example}

\begin{theorem}\label{a3}  {\rm Let $\mathfrak{D}=(X,\tau_{_1},\tau_{_2})$ be a bitopological space
and $\precsim$ a preorder on $X$. Then,
\par
($\mathfrak{a}$) The preorder is $\tau_{_1}\times \tau_{_2}$-closed if and only if for every two points $a, b\in X$ such that 
$a\not\precsim b$ there exist a 
increasing $\tau_{_1}$-neighborhood $V_{_a}$ of $a$ and a 
decreasing $\tau_{_2}$-neighborhood $V_{_b}$ of $b$ which are disjoint.
\par
($\mathfrak{b}$) If the preorder is $\tau_{_1}\times \tau_{_2}$-closed, then
for every $a\in X$ the set $d(a)$ is $\tau_{_1}$-closed and 
the set $i(a)$ is $\tau_{_2}$-closed.
\par
($\mathfrak{c}$)} If the preorder is $\tau_{_1}\times \tau_{_2}$-closed, then
$ a\in cl_{_{\tau_{_1}}}\{b\}\Rightarrow
a\precsim b$ and $ b\in cl_{_{\tau_{_2}}}\{a\}\Rightarrow
a\precsim b$.
\end{theorem}
\begin{proof}
$\mathfrak{a})$ Let $\precsim$ is a $\tau_{_1}\times \tau_{_2}$-closed
and
$a, b\in X$ such that $a\not\precsim b$. 
Since $(a,b)$ does not belong to the graph $G_{_{\precsim}}$ of $\precsim$ and $G_{_{\precsim}}$ is $\tau_{_1}\times \tau_{_2}$-closed, there exist a  $\tau_{_1}$-open neighborhood $V_a$ of $a$ and
a $\tau_{_2}$-open neighborhood $V_b$ of
$b$ such that
$V_a\times V_b\subset X\times X\setminus G_{_{\precsim}}$.
Let $\widetilde{V}_a=\{\lambda\in X\vert \exists \mu\in V_a\ \ {\rm such\ that}\ \ \mu\precsim \lambda\}$ and $\widetilde{V}_b=\{\kappa\in X\vert \exists \nu\in V_b\ \ {\rm such\ that}\ \ \kappa\precsim \nu\}$.
Clearly, $\widetilde{V}_a$ is an increasing $\tau_{_1}$-neighborhood of $a$
and
$\widetilde{V}_b$ is an decreasing $\tau_{_2}$-neighborhood of $b$ 
such that
$ \widetilde{V}_a\cap \widetilde{V}_b=\emptyset$. Indeed,
suppose to the contrary that it is not the case and let
$z\in   \widetilde{V}_a\cap \widetilde{V}_b$. 
Then, $\mu\precsim z$ and $z\precsim \nu$ implies that $\mu\precsim \nu$. Hence,
$(\mu,\nu)\in G_{_{\precsim}}$ and 
$(\mu,\nu)\in  V_{_a}\times V_{_{b}}\subset X\times X\setminus G_{_{\precsim}}$ which is impossible.

Conversely, suppose to the contrary that 
for each $a, b\in X$, $a\not\precsim b$ implies that
there exists 
an increasing $\tau_{_1}$-neighborhood $V_a$ of $a$
and a decreasing $\tau_{_2}$-neighborhood $V_b$ of $b$ 
such that
$ V_a\cap V_b=\emptyset$,
while $\precsim$ is not $\tau_{_1}\times \tau_{_2}$-closed.
Thus, there exists
$(a,b)\in X\times X$ such that $(a,b)\in X\times X\setminus G_{_{\precsim}}$ and
$(a,b)\in cl_{_{\tau_{_1}\times\tau_{_2}}}G_{_{\precsim}}$, or equivalently,
$a\not\precsim b$ and for each
$\tau_{_1}$-neighborhood $V_a$ of $a$
and each $\tau_{_2}$-neighborhood $V_b$ of $b$ there holds $V_a\times V_b\cap G_{_{\precsim}}
\neq\emptyset$. If 
$V_a$ is increasing and $V_b$ is decreasing, then
from $V_a\times V_b\cap G_{_{\precsim}}\neq \emptyset$, we conclude that there exist $x, y\in X$ such that 
$(x,y)\in V_a\times V_b$ and $x\precsim y$.
It follows that $x, y\in V_a\cap V_b$, a contradiction
to our assumption that $V_a\cap V_b$ must be empty for at least one pair of these neighborhoods. 
Therefore, $\precsim$ is $\tau_{_1}\times \tau_{_2}$-closed.
\par\noindent
$\mathfrak{b})$ Let $a\in X$. To show that $d(a)$ is $\tau_{_1}$-closed we prove that
$X\setminus d(a)$ is $\tau_{_1}$-open. Indeed, let $b\in X\setminus d(a)$.
It follows that $b\not\precsim a$. By the first part of proposition
there exists 
an increasing $\tau_{_1}$-neighborhood $\widetilde{V}_b$ of $b$
and a decreasing $\tau_{_2}$-neighborhood $\widetilde{V}_a$ of $a$ 
such that
$ \widetilde{V}_b\cap \widetilde{V}_a=\emptyset$. Hence, there exists a $\tau_{_1}$-open neighborhood $V_b$ of $b$ such that $b\in V_b\subseteq \widetilde{V}_b$ and 
$d(a)\subseteq \widetilde{V}_a$. It follows that $d(a)\cap V_b=\emptyset$ which implies
that $X\setminus d(a)$ is $\tau_{_1}$-open ($b\in V_b\subset X\setminus d(a)$). 
\par\noindent
$\mathfrak{c})$ Let
$a\in cl_{_{\tau_{_1}}}\{b\}$. Then, $b$ belong to all the $\tau_{_1}$-open neighborhoods
$V_a$ of $a$. Suppose to the contrary that $a\not\precsim b$. Then, by the part 
($\mathfrak{a})$
of proposition,
there exist a $\tau_{_1}$-neighborhood $ \widetilde{V}_a$ of $a$
and a
$\tau_{_2}$-neighborhood $ \widetilde{V}_b$ of $b$ 
such that
$ \widetilde{V}_a\cap \widetilde{V}_b=\emptyset$.
By definition, there exists a $\tau_{_1}$-open neighborhood $V_a$ of $a$ such that
$V_a\subseteq \widetilde{V}_a$.
It follows that $b\notin V_a$, a contradiction. Therefore, $a\precsim b$. Similarly we prove that 
$ b\in cl_{_{\tau_{_2}}}\{a\}\Rightarrow
a\precsim b$.
\end{proof}

\begin{proposition}\label{pan}{\rm Each bitopological space $\mathfrak{D}=(X,\tau_{_1},\tau_{_2})$
equipped with a closed preorder is 
a pairwise Hausdorff space; that is $\tau_1\bigvee \tau_2$ is a Hausdoff topology. }
\end{proposition}
\begin{proof}
Take two distinct points into consideration $x, y\in X$. Due to the fact that we are dealing with a preorder $\precsim$, one of the two relationships $x\precsim y$,  $y\precsim x$
is false. Assume the first is false (the case of the second is analogous).
By Theorem \ref{a3},
there exist a 
increasing $\tau_{_1}$-neighborhood $V_{_x}$ of $x$ and a 
decreasing $\tau_{_2}$-neighborhood $V_{_y}$ of $y$ which are disjoint. Since the dual space
$\mathfrak{X}^{\ast}=(X,\tau_{_2},\tau_{_1},(\precsim)^{-1})$,  is also a bitopological preordered space, 
there exist a 
increasing $\tau_{_2}$-neighborhood $V_{_y}$ of $y$ and a 
decreasing $\tau_{_1}$-neighborhood $V_{_x}$ of $x$ which are disjoint. Therefore, 
$\mathfrak{D}$ is a pairwise Hausdorff space, $\tau_1\bigvee \tau_2$ is a Hausdoff topology.
\end{proof}

In \cite{and} Andrikopoulos presents a version of Nachbin's theory based on bitopological ordered spaces. The following theorem is taken from Andrikopoulos \cite[Theorem 1.8]{and}.

\begin{theorem}\label{a231}{\rm A bitopological ordered space $\mathfrak{X}=(X,\tau_{_1},\tau_{_2},\precsim)$
is pairwise normally ordered if and only if given a decreasing $\tau_{_1}$-closed
set $A$ and an increasing $\tau_{_2}$-closed
set $B$ with $A\cap B=\emptyset$, there exists an increasing real-valued function $f$ on $X$ such that
\par
(i) $f(A)=0$, $f(B)=1$ and $0\leq f(x)\leq 1$,
\par
(ii) $f$ is 
$\tau_{_1}$-lower semicontinuous and 
$\tau_{_2}$-upper semicontinuous.}
\end{theorem}

\begin{proof} To prove sufficiency, we follow the line of the proof of the classical theorems of Nachbin \cite[Theorem 2]{nac} and Kelly \cite[Theorem 2.7]{kel} (see also \cite[Theorem 1.8]{and}. 
Let $A$ be
a decreasing $\tau_{_1}$-closed
set in $X$ and $B$ be an increasing $\tau_{_2}$-closed
set $B$ in $X$ with $A\cap B=\emptyset$. 
Since $\mathfrak{X}$ is pairwise normal there exist an 
increasing 
$\tau_{_1}$-open set $O_{_1}$ and a
decreasing $\tau_{_2}$-open set $O_{_2}$ such that $A\subset O_{_2}$, $B\subset O_{_1}$and $O_{_1}\cap O_{_2}=\emptyset$. We put 
$A_{_0}=A$, $A_{_{1\over 2}}=X\setminus O_{_1}$, $B_{_{1\over 2}}=O_{_2}$ and 
$B_{_1}=X\setminus B$.
Then, we have 
$A_{_0}\subseteq B_{_{1\over 2}}\subseteq A_{_{1\over 2}}\subseteq B_{_1}$, $cl_{_{\tau_{_1}}}B_{_{1\over 2}}
\subseteq B_{_1}$. When we apply our hypothesis on $\mathfrak{X}$ to each pair of sets $A_{_0}=A, B_{_{1\over 2}}$
and $A_{_{1\over 2}}, B_{_1}$,
we get decreasing $\tau_{_1}$-closed sets $A_{_{1\over 4}}, A_{_{3\over 4}}$ 
and increasing $\tau_{_2}$-oped sets $B_{_{1\over 4}}, B_{_{3\over 4}}$ 
such that
\begin{center}
$A_{_0}\subseteq B_{_{1\over 4}}\subseteq A_{_{1\over 4}}\subseteq B_{_{1\over 2}}\subseteq A_{_{1\over 2}}\subseteq 
B_{_{3\over 4}}\subseteq A_{_{3\over 4}}\subseteq B_{_1}$.
\end{center}
Continuing this process, keeping in mind that dyadic rationals are dense in R,
we obtain by induction families
$(A_{_s})_{_{s\in S}}$, $(B_{_s})_{_{s\in S}}$, where 
$S=\{\frac{p}{2^{^q}}\vert \ p=1,2,..., 2^q-1,\ q\in \mathbb{N}\setminus \{0\}\}$. 
We put $A_{_s}=\emptyset$ if $s<0$, $A_{_s}=X$ if $s\geq 1$ and 
$B_{_s}=\emptyset$ if $s\leq 0$, $B_{_s}=X$ if $s>1$.
Then,
\begin{center}
$B_{_r}\subseteq B_{_s}\subseteq A_{_s}\subseteq B_{_t}$ if $r\leq s\leq t$, and $A_{_s}\subseteq B_{_t}$ if $s<t$.
\end{center}
Define 
\begin{center}
$f(x)=inf\{t\in S\vert x\in B_{_t}\}$.
\end{center}
Then,
\begin{center}
$f(x)=inf\{t\in S\vert x\in A_{_t}\}$.
\end{center}

We have $f(A_{_0})=f(A)=0$, $f(B_{_1})=f(X\setminus B)=1$ and $0\leq f(x)\leq 1$ for all
$x\in X$.
To prove that $f$ is increasing, let $x\precsim y$ for some $x, y\in X$. 
As a result, for some $t^{\ast}$ in $S$, $y\in B_{_{t^{\ast}}}$ holds true.
 Since $B_{_{t^{\ast}}}$ is decreasing we have that 
$x\in B_{_{t^{\ast}}}$.
By the definition of $f$ we have $f(x)=inf\{t\in S\vert x\in B_{_{t}}\}\leq t^{\ast}$.
Due to the arbitrariness of $t^{\ast}$, we conclude that
$f(x)\leq inf\{t^{\ast}\in S\vert y\in B_{_{t^{\ast}}}\}=f(y)$.

It remains to prove that $f$ is a 
$\tau_{_1}$-lower semicontinuous 
$\tau_{_2}$-upper semicontinuous function on $\mathfrak{X}$.
By the above construction, we have
\begin{center}
$A_{_t}\subseteq int_{_{\tau_{_2}}}A_{_s}$ and $cl_{_{\tau_{_1}}}B_{_t}\subseteq B_{_s}$ for $t<s$. 
\end{center}
To prove that $f$ is a 
$\tau_{_1}$-lower semicontinuous, retaining the previous proof process's symbolism, 
let $t, r, s\in S$ and $\varepsilon>0$ such that $f(x)-\varepsilon<t<r<s<f(x)$.
We know that $x\notin B_{_s}$, so $x\in X\setminus cl_{_{\tau_{_1}}}B_{_r}=G$.
For each $t<r$ and $y\in G$ we have that $y\notin B_{_t}$ and thus $f(x)-\varepsilon<r<f(y)$.
It follows that  $f$ is a 
$\tau_{_1}$-lower semicontinuous function.
To prove that $f$ is a 
$\tau_{_2}$-upper semicontinuous,
let $t, s\in S$ and $\varepsilon>0$ such that $f(x)<t<s<f(x)+\varepsilon$.
We have that $x\in A_{_t}\subseteq int_{_{\tau_{_2}}}A_{_s}$. If $y \in int_{_{\tau_{_2}}}A_{_s}$,
then $f(y)<s<f(x)+\varepsilon$ which completes the proof.

To prove necessity, let
$A, B$ be two disjoint subsets of $X$ such that $A$ is $\tau_{_1}$-closed and $B$ is $\tau_{_2}$-closed.  By hypothesis, there exists an increasing  
$\tau_{_1}$-lower semicontinuous and $\tau_{_2}$-upper semicontinuous function $f$ such that
$f(A)=0, f(B)=1$ and $0\leq f(x)\leq 1$ for each $x\in X$. By the $\tau_{_1}$-lower semicontinuouity of $f$, the set 
$U=f^{-1}((\frac{1}{2},1])$ is $\tau_{_1}$-open such that $B\subset U$ and 
by the $\tau_{_2}$-upper semicontinuouity of $f$, the set 
$V=f^{-1}([0,\frac{1}{2}))$ is $\tau_{_2}$-open set
such that $A\subset V$. 
Clearly, $U\cap V=\emptyset$. On the other hand,
since $f$ is increasing we have that $V$ is decreasing and $U$ is increasing.
\end{proof}

In case where the order considered is
the equality relation 
$\Delta=\{(x,x)\vert x\in X\}$, Theorem \ref{a231} reduces to Kelly's theorem \cite[Theorem 2.7]{kel}.

\begin{proposition}\label{a4}  {\rm Let $\mathfrak{X}=(X,\tau_{_1},\tau_{_2},\precsim)$ be a bitopological 
ordered space. 
\par
($\mathfrak{a}$) If $A$ is a $\tau_{_1}$-compact subset of $X$ and  
$B$ is a $\tau_2$-compact subset of $X$, then the increasing subset $i(A)$ generated 
by $A$ is $\tau_{_2}$-closed
and the decreasing subset $d(B)$ generated by $B$ is $\tau_{_1}$-closed.
\par
($\mathfrak{b}$) If $\tau_{_1}\vee \tau_{_2}$ is quasicompact, then each $\tau_{_1}$-closed subset of $X$ is $\tau_{_2}$-quasicompact and each $\tau_{_2}$-closed subset of $X$ is $\tau_{_1}$-quasicompact.
}
\end{proposition}
\begin{proof}($\mathfrak{a}$) To show that $i(A)$ is $\tau_{_2}$-closed, it is enough to prove that 
$X\setminus i(A)$ is $\tau_{_2}$-open. Indeed, let $b\in  X\setminus i(A)$. Then, 
for each $a\in A$, we have $a\not\precsim b$.
By Theorem \ref{a3}($\mathfrak{a}$), there is an increasing $\tau_{_1}$-neighborhood 
$\widetilde{V}_{_a}$ of $a$ and a decreasing $\tau_{_2}$-neighborhood 
$\widetilde{V}_{_{b(a)}}$ of $b$ such that
$ \widetilde{V}_{_a}\cap \widetilde{V}_{_{b(a)}}=\emptyset$.
On the other hand, for each $a\in A$ there exists an $\tau_{_1}$-open neighborhood $V_a$ of
$a$ such that $V_a\subseteq \widetilde{V}_{_a}$. 
Let $\mathcal{C}=\{V_{_a}\vert a\in A\}$. Then, $\mathcal{C}$ is a $\tau_1$-open cover of $A$.
Therefore, by $\tau_{_1}$-compactness of $A$, there exist $a_{_1}, a_{_2},..., a_{_n}\in A$ such that 
$A\subset V_{_{a_{_1}}}\cup V_{_{a_{_2}}}...
\cup V_{_{a_{_n}}}
\subset \widetilde{V}_{_{a_{_1}}}\cup \widetilde{V}_{_{a_{_2}}}...
\cup \widetilde{V}_{_{a_{_n}}}$.
Let
$ \widetilde{V}_{_{b}}= \widetilde{V}_{_{b(a_{_1})}}\cap \widetilde{V}_{_{b(a_{_2})}}...
\cap \widetilde{V}_{_{b(a_{_n})}}$. Then, $b\in \widetilde{V}_{_{b}}$
and 
\begin{center}
$\widetilde{V}_{_{b}}\cap A\subseteq \widetilde{V}_{_{b}}\cap
(\displaystyle\bigcup_{i\in \{1,2,...,n\}} \widetilde{V}_{_{a_{_i}}})=
\displaystyle\bigcup_{i\in \{1,2,...,n\}} \widetilde{V}_{_{b}}\cap\widetilde{V}_{_{a_{_i}}}
\subseteq \displaystyle\bigcup_{i\in \{1,2,...,n\}} \widetilde{V}_{_{b(a_{_i})}}\cap\widetilde{V}_{_{a_{_i}}}=\emptyset.
$
\end{center}
Therefore, $A\subset X\setminus \widetilde{V}_{_{b}}$. Since 
$X\setminus \widetilde{V}_{_{b}}$ is increasing, we conclude that 
$i(A)\subset X\setminus \widetilde{V}_{_{b}}$ or equivalently 
$i(A)\cap \widetilde{V}_{_{b}}=\emptyset$. Since $\widetilde{V}_{_{b}}$ is a 
$\tau_{_2}$-neighborhood of $b$, there exists a 
$\tau_{_2}$-open neighborhood $W_{_b}$ of $b$ such that $b\in W_{_b}\subseteq \widetilde{V}_{_{b}}$. It follows that 
$i(A)\cap W_{_{b}}=\emptyset$ which implies that  $i(A)$ is 
$\tau_{_2}$-closed. In a similar way we can prove that the subset
$d(A)$ generated by $A$ is $\tau_{_1}$-closed.
\par\smallskip\par\noindent
($\mathfrak{b}$) Let $A$ be a $\tau_{_1}$-closed subset of $X$, then it is $\tau_{_1}\vee \tau_{_2}$-closed.
Since $\tau_{_1}\vee \tau_{_2}$ is quasicompact, we have that 
 $A$ be a $\tau_{_1}\vee \tau_{_2}$-quasicompact subset of $X$. Then,  $A$ is quasicompact
in the weaker topology $\tau_{_2}$. 
\end{proof}

\section{Joincompact bitopological ordered spaces}

We can now proceed to the definition of joincompact bitopological spaces.

\begin{definition}\label{plk}{\rm A bitopological space $\mathfrak{D}=(X,\tau_{_1},\tau_{_2})$
is
{\it joincompact} if it is pairwise Hausdorff and the topology $\tau_{_1}\bigvee \tau_{_2}$ is quasi-compact.} 
\end{definition}

Joincompact spaces and
Lawson-closed subsets of continuous lattices are the same objects with different names. 
In addition, the order on a continuous lattice corresponds to the specialization order of the Scott 
topology of the lattice (see paragraph 6 below).

The following proposition is an immediate consequence of Definition \ref{plk} and Proposition \ref{pan}.

\begin{proposition} {\rm A quasi-compact bitopological preordered space $\mathfrak{X}=(X,\tau_{_1},\tau_{_2},$
$\precsim)$
is a joincompact bitopological ordered space.}
\end{proposition}

\begin{corollary}\label{a5}{\rm  Let $\mathfrak{X}=(X,\tau_{_1},\tau_{_2},\precsim)$ be a joincompact bitopological ordered space. Then,
the $\tau_{_1}$-closed sets are precisely the 
$\tau_{_2}$-compact sets and the $\tau_{_2}$-closed sets are precisely the 
$\tau_{_1}$-compact sets.
}
\end{corollary}

\begin{proposition}\label{a117} {Let $\mathfrak{X}=(X,\tau_1,\tau_2,\precsim)$ be a joincompact bitopological ordered space. 
If $A\subset X$ is a decreasing set and $O_{_1}$ is a $\tau_{_2}$-neighborhood 
of $A$, then there exists a $\tau_{_2}$-open decreasing neighborhood $O_{_2}$
of $A$ such that $A\subset O_{_2}\subset O_{_1}$.}
\end{proposition}
\begin{proof}Let $A$ and $O_{_1}$
be as in the supposition of the proposition.
Put 
$O_{_2}=X\setminus i(cl_{_{\tau_{_2}}}(X\setminus O_{_1}))$. Then, 
$O_{_2}$ is a decreasing $\tau_{_2}$-open subset of $X$. 
By Corollary \ref{a5},  $cl_{_{\tau_{_2}}}(X\setminus O_{_1})$ is $\tau_{_1}$-compact and
by Theorem \ref{a4}($\mathfrak{a}$) $ i(cl_{_{\tau_{_2}}}(X\setminus O_{_1}))$ is $\tau_{_2}$-closed. It follows that $O_{_2}$ is $\tau_{_2}$-open. Now, we have
\begin{center}
$X\setminus O_{_1}\subseteq cl_{_{\tau_{_2}}}(X\setminus O_{_1})\subseteq i(cl_{_{\tau_{_2}}}(X\setminus O_{_1}))$.
\end{center}
We show that $A\subset O_{_2}\subset O_{_1}$. Clearly, $O_{_2}\subset O_{_1}$. To prove that $A\subset O_{_2}$ it suffices to show that $A\cap i(cl_{_{\tau_{_2}}}(X\setminus O_{_1}))=\emptyset$. Indeed, suppose to the contrary that there exists $z\in X$ such that
$z\in A$ and $z\in i(cl_{_{\tau_{_2}}}(X\setminus O_{_1}))$. Then, there exists $w\in 
i(cl_{_{\tau_{_2}}}(X\setminus O_{_1}))$ such that $w\precsim z$. From $z\in A$ we get $w\in A$, a contradiction because $O_{_1}$ is a $\tau_{_2}$-neighborhood of $A$.
\end{proof}

The following proposition 
is the dual of Proposition \ref{a117} and the proof is similar to the one given in
Proposition \ref{a117}.

\begin{proposition}\label{a217} {Let $\mathfrak{X}=(X,\tau_1,\tau_2,\precsim)$ be a joincompact bitopological ordered space. 
If $B\subset X$ is an increasing set and $F_{_1}$ is a $\tau_{_1}$-neighborhood 
of $B$, then there exists a $\tau_{_1}$-open increasing neighborhood $F_{_2}$
of $B$ such that $B\subset F_{_2}\subset F_{_1}$.}
\end{proposition}

\begin{theorem}\label{a7}{\rm Every joincompact bitopological ordered space is
pairwise normally ordered space.}
\end{theorem}
\begin{proof} Consider a decreasing $\tau_{_1}$-closed set $A\subset X$ and an increasing $\tau_{_2}$-closed set $B\subset X$ such that $A\cap B=\emptyset$. Let $b\in B$.
Then, for each $x\in A$ we have $b\not\precsim x$. By Theorem \ref{a3}($\mathfrak{a}$),
there exists a decreasing $\tau_{_2}$-neighborhood $V_{_x}$ of $x$ and an increasing 
$\tau_{_1}$-neighborhood $V_{_b}$ of $b$ such that $V_{_x}\cap V_{_b}=\emptyset$.
By theorem \ref{a3}($\mathfrak{b}$), $d(x)$ is $\tau_{_1}$-closed and 
$i(b)$ is $\tau_{_2}$-closed. Since $\mathfrak{X}$ is joincompact, by \cite[Theorem 3.6(c)]{kop}
we have that
$\mathfrak{X}$ is pairwise normal. Thus, there exist
a $\tau_{_1}$-open set $W$ and 
a $\tau_{_2}$-open set $O$ such that $d(x)\subset O$, $i(b)\subset W$ and $O\cap W=\emptyset$. Using the fact that $d(x)$ is decreasing and applying Proposition \ref{a117},
we obtain a decreasing $\tau_{_2}$-open set $\widetilde{O}$ such that 
$d(x)\subset \widetilde{O}\subset O$. By a dual argument from Proposition \ref{a217},
we can get an increasing $\tau_{_1}$-open set $\widetilde{W}$ such that 
$i(b)\subset \widetilde{W}\subset W$. Since $b\not\precsim x$ for any $x\in A$,
we can find a decreasing $\tau_{_2}$-open set $O_{_x}$ containing $x$ and 
an increasing $\tau_{_1}$-open set $W_{_b}$ containing $b$ such that 
$O_{_x}\cap W_{_b}=\emptyset$. Since $\mathfrak{X}$ is joincompact and $A$ is
$\tau_{_1}$-closed, by Proposition \ref{a4}($\mathfrak{d}$), we have that $A$ is $\tau_{_2}$-quasicompact. Therefore, there exists a finite number of $\tau_{_2}$-open sets $O_{_{x_{_i}}}$, $i=\{1,2,...,n\}$ such that $x_{_i}\in A$ and $A\subseteq \displaystyle\bigcup_{_{i\in \{1,2,...,n\}}}O_{_{x_{_i}}}=\widehat{O}$.
Let $\widehat{W}=\{\displaystyle\bigcap_{_{i\in\{1,2,...,n\}}}W^{x_{_i}}_{_b}\vert O_{_{x_{_i}}}\cap W^{x_{_i}}_{_b}=\emptyset\}.$ It is then clear that $\widehat{W}$ is an increasing $\tau_{_1}$-open set containing $b$ and $\widehat{O}$ is a decreasing $\tau_{_2}$-open set containing $A$  such that
$\widehat{O}\cap \widehat{W}=\emptyset$.
Now, for each $y\in B$, this point $y$ and the set $A$ are in exactly
the same relation as $b$ and $x$ in the preceding procedure. Similarly, 
we find a $\tau_{_1}$-open set $O_{_1}$ and 
a $\tau_{_2}$-open set $O_{_2}$ such that $A\subset O_{_1}$, $B\subset O_{_2}$
and $O_{_1}\cap O_{_1}=\emptyset$.
\end{proof}

\begin{proposition}\label{a112} {Let $\mathfrak{X}=(X,\tau_1,\tau_2,\precsim)$ be a joincompact bitopological ordered space and let $a, b\in X$ such that
$a\not\precsim b$. Then, there exists an increasing $\tau_1$-lower and $\tau_2$-upper semicontinous function $f$ on $X$ such that $f(a)>f(b)$.}
\end{proposition}
\begin{proof} By Theorem \ref{a3}($\mathfrak{b}$) we have that $d(b)$ is a decreasing
$\tau_{_1}$-closed set and 
$i(a)$ is an increasing
$\tau_{_2}$-closed set. But then, Theorems \ref{a231} and \ref{a7} imply that
there exists a increasing
$\tau_{_1}$-lower semicontinuous and 
$\tau_{_2}$-upper semicontinuous function such that $f(b)=0<1=f(a)$.
\end{proof}

\begin{proposition}\label{a2}{\rm Let $\mathfrak{X}=(X,\tau_{_1},\tau_{_2},\precsim)$
be a bitopological ordered space and let $\mathfrak{I}=(\mathbb{R},\mathfrak{U},\mathfrak{L},\leq)$,
where $\mathfrak{U}$, $\mathfrak{L}$ are the upper and lower topologies in $\mathbb{R}$, and $\leq$ 
is the usual order in $\mathbb{R}$. 
If $\mathfrak{X}$ is joincompact and 
$f:\mathfrak{X}\to \mathfrak{I}$
is $\tau_{_1}$-lower and $\tau_{_2}$-upper semicontinuous
function, then there exist points at which the function takes its maximum and minimum value.
}
\end{proposition}
\begin{proof} Let's assume that $f$ is $\tau_{_1}$-lower and $\tau_{_2}$-upper semicontinuous
function on $X$.
Since $\mathfrak{X}$ is joincompact,
$f(X)$ is joincompact and hence it is compact in the weaker 
topologies $\mathfrak{U}$ and $\mathfrak{L}$. Therefore, 
$f(X)$ is $\mathfrak{U}$-closed and $\mathfrak{L}$-closed. Since the topologies $\mathfrak{U}$ and $\mathfrak{L}$
are compact $f$ is bounded, that is, $m\leq f(x)\leq M$ for all $x\in X$ and $m, M\in R$.
Let $\lambda$ be the infimum of set of values of $f(x)$ on $X$. Suppose to the contrary that $f(x)\neq \lambda$ for all $x\in X$.
Let $h(x)=\displaystyle\frac{1}{f(x)-\lambda}$ ($h(x)>0$). 
Then,  $h$ is 
$\tau_{_1}$-lower and $\tau_{_2}$-upper semicontinuous
function on $X$. As in the case of $f$, we conclude that $h$ is bounded. 
Because $\lambda$ is the infimum of $f(x)$ values, given $\epsilon>0$, we can find a value $f(x)<\lambda+\epsilon$.
Hence, $h(x)=\displaystyle\frac{1}{f(x)-\lambda}>\displaystyle\frac{1}{\epsilon}$.
Since $\epsilon$ is arbitrary the inequality 
$h(x)>\displaystyle\frac{1}{\epsilon}$ shows that $h(x)$ is unbounded in $X$, a contradiction.
Therefore,
$\lambda=f(x)$ for some $x\in X$.
Similarly, we conclude that if $\mu=\sup f(X)$ then $\mu=f(x)$ for some $x\in X$.
\end{proof}

Based on the previous proposition, 
in the following where the space $\mathfrak{X}=(X,\tau_{_1},\tau_{_2},\precsim)$ is joincompact, it will also be valid that the 
$\tau_{_1}$-lower and $\tau_{_2}$-upper semicontinuous
functions defined in it are bounded.

\section{Main result}

\begin{definition}\label{a222}{\rm
A function $u$ on $\mathfrak{X}=(X,\tau_{_1},\tau_{_2},\precsim)$ is said to be a
{\it Richter-Peleg utility representation function} for a preorder $\precsim$ on $\mathfrak{X}$ if 
$x \precsim y$ implies $u(x)\leq u(y)$ and 
$x \prec y$ implies $u(x)<u(y)$, where $\prec$ stands for the 
strict part of $\precsim$.
A {\it Richter-Peleg multi-utility representation} $\mathcal{V}$ for a preorder $\precsim$ on 
$\mathfrak{X}$ is a 
multi-utility representation for $\precsim$ such that every function $u\in V$ is a 
Richter-Peleg utility for $\precsim$.
A Richter-Peleg multi-utility representation $\mathcal{V}$
is {\it maximal} if none of its proper supersets has this property.
}
\end{definition}

\begin{proposition}\label{a111} {\rm Let $\mathfrak{X}=(X,\tau_1,\tau_2,\precsim)$ be a joincompact bitopological ordered space. 
Then, $\precsim$ has a
$\tau_1$-lower and $\tau_2$-upper semicontinous Richter-Peleg utility representation.}
\end{proposition}
\begin{proof} By Theorems \ref{a231}, \ref{a7} and Proposition \ref{a2}, $\mathfrak{X}$ has an increasing 
real-valued function $f$ on $X$ such that
$f$ is 
$\tau_{_1}$-lower semicontinuous and 
$\tau_{_2}$-upper semicontinuous. Therefore, if $x\precsim y$ the we have $f(x)\leq f(y)$.
On the other hand, if $x\prec y$, then $y\not\precsim x$ and by Proposition
\ref{a112} we have that $f(x)<f(y)$. It follows that 
$f$ is a $\tau_1$-lower and $\tau_1$-upper semicontinous Richter-Peleg utility representation of $\precsim$.
\end{proof}

According to Proposition \ref{a2}, 
a Richter-Peleg multi-utility representation $\mathcal{V}$ of $\precsim$ in a joincompact bitopological 
ordered space
$\mathfrak{X}=(X,\tau_1,\tau_2,\precsim)$
consists of bounded functions. In what follows, $\mathcal{L}_{_1}\mathcal{U}_{_2}(\mathfrak{X})$
(resp. $\mathcal{B}\mathcal{L}_{_1}\mathcal{U}_{_2}(\mathfrak{X})$)
denotes the family of (resp. bounded) $\tau_{_1}$-lower and 
$\tau_{_2}$-upper semicontinuous
functions of $\mathfrak{X}$ to 
$\mathfrak{I}=(\mathbb{R},\mathfrak{U},\mathfrak{L},\leq)$
and 
$\mathcal{R}\mathcal{P}\mathcal{B}\mathcal{L}_{_1}\mathcal{U}_{_2}(\mathfrak{X})$
denotes the members of $\mathcal{B}\mathcal{L}_{_1}\mathcal{U}_{_2}(\mathfrak{X})$
which are Richter-Peleg utility representations of $\precsim$.
Define the {\it supremum norm} on 
$\mathcal{B}\mathcal{L}_{_1}\mathcal{U}_{_2}(\mathfrak{X})$ as
\begin{center}
$\norm{f}_{_\infty}=\sup \{|f(x)|\ \vert\ x\in X\}$.
\end{center}
For each subfamily $\mathcal{V}$ of $\mathcal{L}_{_1}\mathcal{U}_{_2}(\mathfrak{X})$,
$cl_{_{\mathcal{U}n}}\mathcal{V}$ denotes the closure of 
$\mathcal{V}$
with respect to the usual (uniform) norm topology.

The next proposition is a dual version of Lemma 3.4 in \cite{kel}.

\begin{proposition}\label{lase}{\rm The class $\mathcal{L}_{_1}\mathcal{U}_{_2}(X)$
of real-valued functions which are $\tau_{_1}$-lower semicontinuous and 
$\tau_{_2}$-upper semicontinuous on a space $\mathfrak{D}=(X,\tau_{_1},\tau_{_2})$ is complete
with respect to the uniform norm on $X$.}
\end{proposition}

\begin{definition}\label{da1}{\rm Let $X$ be a set and let $\mathfrak{F}$ be a family of functions each of which has domain $X$. Then the family $\mathfrak{F}$ is {\it separating}
if, for each pair $x, y\in X$ of distinct elements of $X$, there exists an $f_{_{xy}}$ in
$\mathfrak{F}$ such that $f_{_{xy}}(x)\neq f_{_{xy}}(y)$.
}
\end{definition}

\begin{lemma}\label{a098}{\rm  Let $\mathfrak{X}=(X,\tau_1,\tau_2,\precsim)$ be a bitopological 
joincompact ordered space and
let $\mathcal{V}$ be a $\tau_{_1}$-lower and $\tau_{_2}$-upper semicontinuous 
Richter-Peleg multi-utility representation of $\precsim$.
Then, $cl_{_{\mathcal{U}n}}\mathcal{V}$ is also a 
a Richter-Peleg multi-utility representation of $\precsim$.}
\end{lemma}
\begin{proof} Let $f$ belongs to the closure of 
$\mathcal{V}$
with respect to the uniform norm topology on $\mathfrak{X}$. 
Then,
there exists a sequence $(f_{_n})_{_{n\in\mathbb{N}}}\in \mathcal{V}$ such that
$f_{_n}$ converges to $f$ in the uniform norm topology.
We prove that $f$ is a bounded, increasing, order preserving, $\tau_{_1}$-lower and $\tau_{_2}$-upper 
semicontinuous function.
Uniform convergence implies that for any $\varepsilon>0$ there is an $N_{_{\varepsilon}}\in\mathbb{N}$ such that 
$|f_{_n}(x)-f(x)|<\frac{\varepsilon}{3}$
for all $n\geq N_{_{\varepsilon}}$ and all $x\in X$. 
\par\smallskip\noindent
($\mathfrak{a}$) To prove that $f$ is bounded, let $\varepsilon=3$. 
Since each $f_{_n}$ is bounded on $X$, there exists $M_{_n}$ such that $|f_{_n}(x)|\leq M_{_n}$ for all $x\in X$. By uniform convergence, there exists an $N_{_3}\in\mathbb{N}$ such that $|f_{_n}(x)-f(x)|\leq 1$ for all $n\geq N_{_3}$. Then, 
$|f(x)|-|f_{_{N_{_3}}}(x)|\leq |f(x)-f_{_{N_{_3}}}(x)|\leq  1$ which implies that 
$|f(x)|\leq |f_{_{N_{_3}}}(x)|+1\leq M_{_{N_{_3}}}+1$.
\par\smallskip\noindent
($\mathfrak{b}$) To demonstrate that $f$ is a Richter-Peleg utility function, 
we begin by proving that if $a\sim b$ for some $a, b$ in X, then $f(a)=f(b)$.
Suppose to the contrary that 
$f(a)\neq f(b)$.
Since the sequence $(f_{_n})_{_{n\in\mathbb{N}}}$ converges uniformly to $f$, we have that
$f_{_n}(a)\to f(a)$, $f_{_n}(b)\to f(b)$ and for each $n\in \mathbb{N}$, $f_{_n}(a)=f_{_n}(b)$.
Therefore, $f_{_n}(a)$ converges 
to both $f(a)$ and $f(b)$, which contradicts the
Hausdorffness of $(X,\tau)$, where $\tau$ is the usual order topology on $\mathbb{R}$.
To prove that $f$ is order preserving, suppose to the contrary that $f(a)\geq f(b)$
for some $a$ and $b$, where $a\prec b$. Let $\varepsilon=f(a)-f(b)$.
Since the sequence $(f_{_n})_{_{n\in\mathbb{N}}}$ converges uniformly to $f$, 
we have that for each $n\geq N_{_\varepsilon}$ there holds
$\sup \{|f(x)-f_{_n}(x)|\vert x\in X\}<\frac{\varepsilon}{3}$.
It follows that $f(a)-f_{_n}(a)<\frac{\varepsilon}{3}$  and 
$f_{_n}(b)-f(b)<\frac{\varepsilon}{3}$ for $n>n_{_0}$. Therefore,
$f_{_n}(a)-f_{_n}(b)>-\frac{2\varepsilon}{3}+(f(a)-f(b))=\frac{\varepsilon}{3}>0$, 
contradicting the fact that $f_{_n}$ is order preserving.
Hence, $f$ is order preserving.

\par\smallskip\noindent
($\mathfrak{c}$) It remains to prove that
$f$ is $\tau_{_1}$-lower and $\tau_{_2}$-upper semicontinuous function.
Fix $x_{_0}$ in $X$ and let $x\in X$.
For any $n\in \mathbb{N}$ we have
\begin{center}
$f(x)-f(x_{_0})=(f(x)-f_{_n}(x))+(f_{_n}(x)-f_{_n}(x_{_0}))+(f_{_n}(x_{_0})-f(x_{_0})\leq$\\
$2\sup \{|f_{_n}(x)-f(x)|\vert x\in X\}+(f_{_n}(x)-f_{_n}(x_{_0}))$.
\end{center}
Therefore, for any $n$ we have,
\begin{center}
$f(x)-f(x_{_0})\leq 2\sup \{|f_{_n}(x)-f(x)|\vert x\in X\}+(f_{_n}(x)-f_{_n}(x_{_0}))$.
\end{center}
Since the previous inequality is true for an arbitrary value of $n$, we can choose 
a positive number $\varepsilon>0$, as small as we like. 
Therefore, there must be an $N>0$ for which we have 
$\sup \{|f_{_n}(x)-f(x)|\vert x\in X\}<\frac{\varepsilon}{3}$ for all $n>N$.
Because
$f_{_N}(x)$ is $\tau_{_2}$-upper semicontinuous, so for any choice of the number
$\varepsilon>0$ there is $\tau_{_2}$-open neighborhood $\mathcal{M}$ 
such that for all $x\in \mathcal{M}$,
$f_{_N}(x)-f_{_N}(x_{_0})<\frac{\varepsilon}{3}$.
Using this, we now have
\begin{center}
$f(x)-f(x_{_0})\leq 2\sup \{|f_{_n}(x)-f(x)|\vert x\in X\}+(f_{_n}(x)-f_{_n}(x_{_0}))<
2\frac{\varepsilon}{3}+\frac{\varepsilon}{3}=\varepsilon$
\end{center}
for all $x\in \mathcal{M}$. Therefore there exists a $\tau_{_2}$-open neighborhood 
$\mathcal{M}\subseteq X$ of $x_{_0}$ such that
for each $x\in \mathcal{M}$ we have $f(x)<f(x_{_0}+\varepsilon$.
It follows that $f$ is $\tau_{_2}$-upper semicontinous.
Similarly, interchanging the order of the difference of $f(x)$ and $f(x_{_0})$
we can prove that
$f$ is $\tau_{_1}$-lower semicontinous. 
Hence,
$cl_{_{\mathcal{U}n}}\mathcal{V}$ is also
a $\tau_{_1}$-lower and $\tau_{_2}$-upper semicontinous
Richter-Peleg multi-utility representation of $\precsim$ and this complete the proof.

Obviously, if members of $\mathcal{V}$ are bounded, so will members of $cl_{_{\mathcal{U}n}}\mathcal{V}$.
\end{proof}

\begin{proposition}\label{a812} {Let $\mathfrak{X}=(X,\tau_1,\tau_2,\precsim)$ be a 
joincompact bitopological ordered space and $\mathcal{L}$ be
a closed, with respect to uniform norm topology, sublattice of 
$\mathcal{B}\mathcal{L}_{_1}\mathcal{U}_{_2}(\mathfrak{X})$.
Then, a $\tau_1$-lower and $\tau_2$-upper semicontinous function $\phi$
belongs to
 $\mathcal{L}$
if and only if for each
$x, y\in X$ there exists $f_{_{xy}}\in \mathcal{L}$ such that 
$f_{_{xy}}(x)=\phi(x)$ and $f_{_{xy}}(y)=\phi(y)$.}
\end{proposition}
\begin{proof}{\it Necessity}:
If $\phi\in\mathcal{L}$, then $\phi$ itself satisfies the requirements of the proposition.
\par\noindent
{\it Sufficiency}: Suppose that $\phi$ is a $\tau_1$-lower and $\tau_2$-upper semicontinous function
on $\mathfrak{X}$.
Let $x, y \in X$ with $x\neq y$ and let $f_{_{xy}}\in \mathcal{L}$ such that
$f_{_{xy}}(x)=\phi(x)$, $f_{_{xy}}(y)=\phi(y)$. If $\phi \in \mathcal{L}$, then we have nothing to prove.
Otherwise
assume that a
$\phi \notin \mathcal{L}$. So, $f_{_{xy}}$ approximates $\phi$
in neighborhoods around $x$ and $y$. If $\phi(x)=\phi(y)$ then we can take a constant.
If not, since $\mathcal{L}$ is closed,
it suffices to show that for each $\varepsilon>0$ there exists $f\in \mathcal{L}$ such that
for all $z\in X$ we have
\begin{center}
$\varphi(z)-\varepsilon<f(z)<\varphi(z)+\varepsilon$,
\end{center}
or equivalently
\begin{center}
$\sup\{|\varphi(z)-f(z)|\vert z\in X\}<\varepsilon$,
\end{center}
for it will follow from this that $\norm{f-\phi}_{_{\infty}}<\varepsilon$.
Now, fix an $x\in X$, and let $y\in X$ vary. Put
\begin{center}
$U_{_y}=\{z\in X\vert \ f_{_{xy}}(z)<\phi(z)+\varepsilon\}=\{z\in X\vert \ f_{_{xy}}(z)-\phi(z)<\varepsilon\}$.
\end{center}
Since $f_{_{xy}}-\varphi$ is $\tau_{_1}\vee \tau_{_2}$-lower semicontinuous and 
$\tau_{_1}\vee \tau_{_2}$-upper semicontinuous we conclude that
 $f_{_{xy}}-\varphi$ is $\tau_{_1}\vee \tau_{_2}$-continuous.
Then, $U_{_y}$ is $\tau_{_1}\vee \tau_{_2}$-open set because $f_{_{xy}}-\varphi$ is $\tau_{_1}\vee \tau_{_2}$-continuous.
Also, $y\in U_{_y}$ and thus $\{U_{_y}\vert y\in X\}$ is a 
$\tau_{_1}\vee \tau_{_2}$-open cover of $X$. 
Since $\mathfrak{X}$ is joincompact, we have that $X$ is $\tau_{_1}\vee \tau_{_2}$-compact. Hence, there exists finitely many $y_{_1},...,y_{_n}\in X$ such that
\begin{center}
$X=\displaystyle\bigcup_{i\in\{1,...,n\}}U_{_{y_{_i}}}$.
\end{center}
Let $f_{_{xy_{_1}}}, f_{_{xy_{_2}}},...,f_{_{xy_{_n}}}$ are the functions of $\mathcal{L}$ which correspond to the sets $U_{_{y_{_1}}},U_{_{y_{_2}}},$
$...,U_{_{y_{_n}}}$ respectively.
Put
\begin{center}
$g_{_x}=f_{_{xy_{_1}}}\wedge f_{_{xy_{_2}}}\wedge...\wedge f_{_{xy_{_n}}}$.
\end{center}
Since  $\mathcal{L}$ is semi-vector lattice we have that $g_{_x}\in \mathcal{L}$.
On the other hand, $g_{_x}(x)=\phi(x)$ and $g_{_x}(z)<\phi(x)+\varepsilon$ for all
$z\in X$.
We next consider the open set 
\begin{center}
$V_{_x}=\{z\in X\vert \ g_{_x}(z)>\phi(z)-\varepsilon\}$.
\end{center}
Then, $V_{_x}$ is $\tau_{_1}\vee \tau_{_2}$-open set because $g_{_x}$ is 
$\tau_{_1}\vee \tau_{_2}$-continuous. Also, $x\in V_{_x}$ and thus 
$\{V_{_x}\vert x\in X\}$ is a $\tau_{_1}\vee \tau_{_2}$-open cover of $X$. Since $\mathfrak{X}$ is joincompact, we have that $X$ is $\tau_{_1}\vee \tau_{_2}$-compact.
Therefore, this open cover has a finite subcover $\{V_{_{x_{_1}}}, V_{_{x_{_1}}}, ..., V_{_{x_{_n}}}\}$.
We denote the corresponding
functions in 
$\mathcal{L}$ by $g_{_{x_{_1}}}, g_{_{x_{_2}}},...,g_{_{x_{_n}}}$,
and we define $f$ by $f=g_{_{x_{_1}}}\vee g_{_{x_{_2}}}\vee,...,\vee g_{_{x_{_n}}}$.
It is clear that $f\in\mathcal{L}$ with the
property that
\begin{center}
$\varphi(z)-\varepsilon<f(z)<\varphi(z)+\varepsilon$,
\end{center}
for all $z\in X$. Therefore, for each $\varepsilon>0$ we have 
$\norm{f-\phi}_{_{\infty}}<\varepsilon$ and
so our
proof is complete.
\end{proof}

\begin{definition}{\rm Let $\mathcal{L}=(X,\precsim)$ be a lattice. We say that 
$\mathcal{L}$
is a {\it generalized cone lattice} if it is closed under multiplication by scalars and addition
of constants (in short, $f\in \mathcal{L}$ and $a, b\in \mathbb{R}$ imply 
$a f+b\in\mathcal{L}$).}
\end{definition}

\begin{theorem}\label{a119} {Let $\mathfrak{X}=(X,\tau_1,\tau_2,\precsim)$ be a joincompact bitopological ordered space. 
Then, the set of all 
$\tau_{_1}$-lower and $\tau_{_2}$-upper semicontinous
Richter-Peleg utility representations of $\precsim$
is a separating and closed (with respect to uniform norm topology) 
generalized cone lattice of the set of $\tau_{_1}$-lower and $\tau_{_2}$-upper semicontinous
functions on $\mathfrak{X}$.
Conversely, given a joincompact  bitopological space
$\mathfrak{D}=(X,\tau_1,\tau_2)$ and a
separating and closed (with respect to uniform norm topology) 
generalized cone lattice $\mathcal{U}$ of 
$\tau_{_1}$-lower and $\tau_{_2}$-upper semicontinous functions on $\mathfrak{X}$, there is one and only one way
to turn $\mathfrak{D}$ into a joincompact bitopological ordered space $\mathfrak{X}=(X,\tau_1,\tau_2,\precsim)$ where $\mathcal{U}$ is the set of all increasing
$\tau_{_1}$-lower and $\tau_{_2}$-upper semicontinous
Richter-Peleg utility representations of $\precsim$.
}
\end{theorem}
\begin{proof}
Let $\mathfrak{X}=(X,\tau_1,\tau_2,\precsim)$ be a bitopological ordered space  and let 
$\mathcal{V}$ be the set of all 
Richter-Peleg 
utility representation
of
$\precsim$. 
We
will prove that $\mathcal{V}$ satisfies the properties given in the statement of the present proposition. 
\par\smallskip\par\noindent
(i) $\mathcal{V}$ {\it separates points on $X$}.  Indeed, let $x, y\in X$ with $x\not\sim y$. 
Then, $x\not\precsim y$ or $y\not\precsim x$.
By Proposition \ref{a112} we have that $f(y)<f(x)$ or $f(x)<f(y)$ for some
$f\in \mathcal{V}$. Therefore, in any case there exists $f\in \mathcal{V}$
such that $f(x)\neq f(y)$, so $\mathcal{V}$
separates points. 
\par\smallskip\par\noindent
(ii) $\mathcal{V}$ {\it is closed under the uniform norm topology.} 
Let $f\in cl_{_{\mathcal{U}n}}\mathcal{V}$.
By Lemma \ref{a098}
$cl_{_{\mathcal{U}n}}\mathcal{V}\supset \mathcal{V}$
is a Richter-Peleg 
multi-utility representation of $\precsim$, a contradiction to the maximal character of $\mathcal{V}$. Hence, $cl_{_{\mathcal{U}n}}\mathcal{V}=\mathcal{V}$
which implies that $\mathcal{V}$ is closed with respect to uniform norm topology.
\par\smallskip\par\noindent
(iii) $\mathcal{V}$ {\it is a lattice}. Let $f, g \in \mathcal{V}$ and fix 
an $\varepsilon>0$ and a $x_{_0}\in X$. Since $f, g$ are $\tau_{_1}$-lower and $\tau_{_2}$-upper semicontinuous, there are
$\tau_{_1}$-open neighborhods $O, O^{\prime}$
and $\tau_{_2}$-open neighborhods 
$P, P^{\prime}$ such that for all $y\in O, y\in O^{\prime}, y\in P$ and $y\in P^{\prime}$
we have, respectively, that $f(y)<f(x_{_0})+\varepsilon$, 
$f(x_{_0})-\varepsilon <f(y)$,
$g(y)<g(x_{_0})+\varepsilon$ and 
$g(x_{_0})-\varepsilon <g(y)$. 
If $\widetilde{K}=O\cap P$ and 
$\widetilde{K^{\prime}}=O^{\prime}\cap P^{\prime}$, then for each $y\in \widetilde{K}$ we have
that $\sup \{f(y),g(y)\}<\sup\{f(x_{_0}),g(x_{_0})\}+\varepsilon$ and 
for each 
$y\in \widetilde{K^{\prime}}$ we have
that $\inf\{f(x_{_0}),g(x_{_0})\}-\varepsilon<\inf \{f(y),g(y)\}$. 
So, $f\vee g$ is a
$\tau_{_1}$-lower and $\tau_{_2}$-upper semicontinuous function.
On the other hand, since $f, g$ are increasing and order preserving, for each $x_{_1}, x_{_2}\in X$ with
$x_{_1}\precsim x_{_2}$ we have that $\sup \{f(x_{_1}),g(x_{_1})\}\precsim \sup \{f(x_{_2}),g(x_{_2})\}$ and for each $x_{_1}, x_{_2}\in X$ with
$x_{_1}\prec x_{_2}$ we have that $\sup \{f(x_{_1}),g(x_{_1})\}\prec \sup \{f(x_{_2}),g(x_{_2})\}$. 
Thus $f\vee g$ is increasing and order preserving. Since 
$\mathcal{V}$ is maximal with respect to set inclusion, we conclude that 
$f\vee g\in \mathcal{V}$.
Similarly we prove that 
$f\wedge g\in \mathcal{V}$ which implies that $\mathcal{V}$ is a lattice.
\par\smallskip\par\noindent
(iv) $\mathcal{V}$ {\it is a generalized cone lattice}. 
Let
$f$ be an increasing $\tau_{_1}$-lower and $\tau_{_2}$-upper semicontinuous
function on $\mathcal{V}$ and let $\lambda, \kappa\in\mathbb{R}$ with $\lambda\geq 0$.
Clearly, $\lambda f+\kappa$ is increasing. It rermains to prove that
$\lambda f+\kappa\in \mathcal{V}.$
If $\lambda=0$, then $\lambda f+\kappa=\kappa\in \mathcal{V}.$
Otherwise, $\lambda>0$. Fix an $\varepsilon>0$ and an $x_{_0}\in X$.
Since $f$ is $\tau_{_1}$-lower semicontinuous there exists a 
$\tau_{_1}$-open neighborhood $U_{_{x_{_0}}}$ of $x_{_0}$ such that for each 
$y\in U_{_{x_{_0}}}$, $f(y)<f(x_{_0})+\frac{\varepsilon}{\lambda}$ holds.
Let
\begin{center}
$A_{_y}=\{y\in X\vert \lambda f(y)+\kappa<\lambda f(x_{_0})+\kappa+\varepsilon\}$.
\end{center}
Then, $A_{_y}=\{y\in X\vert \lambda f(y)<\lambda f(x_{_0})+\varepsilon\}=
\{y\in X\vert f(y)<f(x_{_0})+\frac{\varepsilon}{\lambda}\}=U_{_{x_{_0}}}$.
It follows that $\lambda f+ \kappa$ is $\tau_{_1}$-lower semicontinuous. Similarly, we prove that $\lambda f+ \kappa$ is $\tau_{_2}$-upper semicontinuous.
\par\smallskip\par\noindent

Conversely, let $\mathcal{U}$ be a
separating and closed (with respect to uniform norm topology) 
generalized cone lattice of $\mathfrak{D}$. The constant functions are in $\mathcal{U}$, so this isn't empty.
Define a order $\precsim$ on $\mathfrak{D}$ as follows:
\begin{center}
$x\sim y$ if $f(x)=f(y)$ and $x\prec y$ if $f(x)<f(y)$
for all $x, y\in X$ and $f\in \mathcal{V}$.
\end{center}
We prove that $\precsim$ is a $\tau_{_1}\times \tau_{_2}$-closed subset of $X\times X$
and thus $\mathfrak{X}=(X,\tau_1,\tau_2,\precsim)$
is a bitopological ordered space. 
Indeed,
Let $a, b\in X$ with $a\not\precsim b$. 
Since $\mathcal{V}$ is separating, there exists at least one 
$f\in\mathcal{U}$ such that $f(a)\neq f(b)$ ($f(a)< f(b)$ or $f(b)< f(a)$).
Let $\mathfrak{F}_{_{a,b}}=\{f\in \mathcal{V}\vert f(a)\neq f(b)\}$.
Then, there exists $f^{\ast}\in \mathfrak{F}_{_{a,b}}$
such that $f^{\ast}(b)<f^{\ast}(a)$, because otherwise, for all $f$ in $\mathcal{V}$ 
there holds $f(a)\leq f(b)$. But then, $a\precsim b$ which is a contradiction to
$a\not\precsim b$.
By the density of reals,
there exists
$r>0$ for which 
$f^{\ast}(b)<r<f^{\ast}(a)$ holds. Let $U_{_b}=(f^{\ast})^{-1}(]-\infty,r[)$ and $U_{_a}=(f^{\ast})^{-1}(]r,\infty[)$. Then,
$U_{_b}$ is a  $\tau_{_2}$-decreasing neighborhood of $b$ and 
$U_{_a}$ is a $\tau_{_1}$-increasing neighborhood of $a$ such that
$U_{_a}\cap U_{_b}=\emptyset$. Hence, by Theorem \ref{a3} we conclude that
$\precsim$ is a $\tau_{_1}\times \tau_{_2}$-closed subset of $X\times X$ which implies that 
$\mathfrak{X}$ is a bitopological ordered space. It remains to prove that $\mathcal{U}$
is the set of all Richter-Peleg utility representations of $\precsim$. 
By definition, all $f\in \mathcal{U}$ are 
Richter-Peleg utility representations of $\precsim$.

Let $\varphi$ be a Richter-Peleg utility representation of $\precsim$ on $\mathfrak{X}$
and let $x, y\in X$. We prove that there exists $h\in \mathcal{U}$
such that $h(x)=\varphi(x)$ and $h(y)=\varphi(y)$, and thus, since  
$\mathcal{U}$ is closed with respect to uniform norm topology, by Proposition \ref{a812} we have that $\varphi\in \mathcal{U}$.
If $\varphi(x)=\varphi(y)=\mu$, then since $\mathcal{U}$ is non-empty,
by taking $\lambda=0$ we observe that the function $h=0.f+\mu$ with $f\in \mathcal{U}$ belongs to $\mathcal{U}$ and $h(x)=\varphi(x)$
and $h(y)=\varphi(y)$.
Now, assume that $\varphi(x)<\varphi(y)$. Then, since $\varphi$ is order preserving we conclude that $y\not\precsim x$. By corollary \ref{a112} there exists an increasing $\tau_{_1}$-lower semicontinuous and $\tau_{_2}$-upper semicontinuous function $f$
such that $f(x)<f(y)$.
Choose the real numbers $\lambda, \kappa$ such that 
$\lambda f(x)+\kappa=\varphi(x)$ and 
$\lambda f(y)+\kappa=\varphi(y)$. Since 
$\lambda=\frac{\varphi(y)-\varphi(x)}{f(y)-f(x)}>0$
and $\mathcal{U}$ is closed with respect to uniform norm topology, 
by Proposition \ref{a812}, we conclude that
$\varphi\in \mathcal{U}$.
It follows that 
$\mathcal{U}$ is the set of all 
$\tau_{_1}$-lower and $\tau_{_2}$-upper semicontinous
Richter-Peleg utility representations of $\precsim$.
\end{proof}

\section{The case of preordered sets}
Let $\mathcal{P}=(X,\precsim)$ be a preordered set.
A subset $D$ of $\mathcal{P}$ is {\it directed} provided it is nonempty, and every finite subset of $D$ has an upper bound in $D$. 
We use $\bigvee^{\mathcal{P}} D$ (resp. $\bigwedge^{\mathcal{P}} D$)
to represent the supremum (resp. infimum) of $D$ if $D$ is a directed set and the supremum (resp. infimum) exists in the preordered set.
A {\it directed
complete preordered set} is a preordered set such that each of
its directed subsets has a supremum. Note that {\it dcpo} normally stands for directed complete partially ordered set (poset).
A directed complete poset which is a lattice is called a {\it directed complete lattice}.
(ii) A {\it complete lattice} $\mathcal{L}$ is a poset in which every subset has a supremum and an infimum.
An {\it ideal} of $\mathcal{P}$ is a directed lower set.
The topology on $\mathcal{P}$ generated by $\{X\setminus \displaystyle\uparrow \{x\}\vert
x\in X\}$
is called the {\it lower topology} and denoted by $\omega(\mathcal{P})=\omega$. 
Let $x, y$ be elements of $X$. 
We say that $x$ is {\it way-below} $y$, written $x\ll y$, if for any directed subset $D$ with
$\bigvee^{\mathcal{P}} D$ exists and $y\precsim \bigvee^{\mathcal{P}} D$, implies that $x\precsim d$ for some $d\in D$.
If  $x\ll y$ then $x\precsim y$ (consider the directed set $D=\{y\}$).
A poset is a {\it continuous poset}  if every element is the join of a directed set of those elements which are way-below it.
A lattice $\mathcal{L}$ is called a {\it continuous lattice} if it is a complete lattice and if every element is the join of those elements which are way-below it.
A subset $A$ of $X$ is called {\it Scott-open} if $A$ is an upper set ($A=\displaystyle\uparrow A$) and for a directed set $D$ with 
$\bigvee^{\mathcal{P}} D\in A$, we have $d\in A$ for some
$d\in D$.
The Scott-open sets satisfy the axioms of a
topology, which we call the {\it Scott topology} and we denote by
$\sigma(\mathcal{P})=\sigma$.
If $\mathcal{P}$ is a continuous poset, then
all sets of the form
$\turnnw{\twoheadleftarrow}{x}=\{y\vert x\ll y\}$
are Scott-open sets. 
The scott topology $\sigma$ on a poset $\precsim$ is always compatible, that is, $\sigma=\precsim_{\sigma}$.
The Scott
topology of the natural order of the real line is the topology of
lower semicontinuity (the topology for which the non-trivial open
sets are the intervals  $(a,\infty)$).
The common refinement $\sigma(\mathcal{P})\vee \omega(\mathcal{P})$ of the Scott and lower topologies
is called the {\it Lawson topology} and is denoted
by $\lambda(\mathcal{P})$(cf. \cite[Definition III-11.5]{GH}).
Given two preordered sets $\mathcal{P}=(X,\precsim)$ and $\mathcal{Q}=(Y,\sqsubseteq)$,
a function $f: \mathcal{P}\to \mathcal{Q}$
is an {\it order embedding} if
for all $x, y\in X$, one has $x\precsim y$ if and only if $f(x) \sqsubseteq f(y)$.
In case of a poset $\mathcal{P}$, this condition forces $f$ to be one-to-one. 
An order embedding is a type of monotone function in order theory that allows one preordered set to be included in another. 
An {\it extension} of a preordered set $\mathcal{P}$
is a pair $(f,\mathcal{Q})$, where $\mathcal{Q}$ is a preordered set and 
$f: \mathcal{P}\to \mathcal{Q}$
is an order embedding. 
A completion of a preordered set $\mathcal{P}$ is an extension 
$(f,\mathcal{Q})$
of
$\mathcal{P}$
such that $\mathcal{Q}$ is a complete lattice.
A subset $I$ of $X$ is an {\it ideal}, if the following conditions hold:
(i) $I$ is non-empty; (ii)
for every $x\in I$, any $y\in X$ and $y\precsim x$ implies that $y\in X$ ($I$ is a lower set), and
for every $x, y\in I$, there is some element $z\in I$, such that 
$x \precsim z$ and $y\precsim z$ ($I$ is a directed set).
The smallest ideal that contains an element $x\in X$ is called a {\it principal ideal}.
This is denoted by
${\displaystyle \downarrow x}=\{y\in X\vert y\precsim x\}$.
Given a subset $A$ of a preordered set $\mathcal{P}=(X,\precsim)$, we denote by 
$A^{\uparrow}$ and $A^{\downarrow}$ the sets of all upper and lower bounds of $A$, respectively.
There are various definitions of a cut of a preordered set $X$ yielding a completion of $X$.
MacNeille \cite{MN} has introduced the famous ``completion by cuts'' for arbitrary preordered sets.
A {\it cut} is a pair $(A,B)$ such that $A=B^{\downarrow}$ and $B=A^{\uparrow}$.
The collection of all cuts, ordered by $(A,B)\leq (C,D)$ if and only if $A\subseteq C$ and $B\subseteq D$
is a complete lattice, called the {\it Dedekind-MacNeille completion} of $\mathcal{P}$.
Any isomorphic copy of Dedekind-MacNeille completion is referred to as the normal completion of $\mathcal{P}$.
Normal completions can be characterized in a number of ways. Because any part of a cut determines the other, 
working with lower cuts is generally more convenient. The normal completion of $\mathcal{P}$ by lower cuts is defined as follows:
Define $A^{\delta}=(A^{\uparrow})^{\downarrow}$,
then the {\it lower cuts completion} of $X$ consists of all subsets $A$ for which
$A^{\delta}=A$.
If $\delta(X)=\{A^\delta\vert A\subseteq X\}$, then
$A^{\delta}\precsim^{\delta}B^{\delta}$ in completion if and only if $A^{\delta}\subseteq B^{\delta}$ as sets. 
The lower cuts completion $(\delta(X), \precsim^{\delta})$ of $\mathcal{P}$
is a complete lattice (\cite[Lemma 1]{abi}).
If $(X,\precsim)$ is a qoset,
then each element $x\in X$  corresponds to its principal ideal
$\downarrow x$
into the lower cuts completion $\delta(X)$. 
The lower cut completion shall be referred to as the {\it normal completion of} $\mathcal{P}$ in the following paragraphs.

Ern\'{e} \cite{ern2} introduced a new way-below relation and 
on the basis of this notion, he defines the concept of precontinuous preordered sets for arbitrary preordered sets. 
The work of Ern\'{e} was influenced by the use of Frink ideals \cite{fri} instead of directed lower sets.

Formally, this notion is defined as follows.
\begin{definition}\label{pas}{\rm (\cite{fri}) A subset $I$ of a preordered set $\mathcal{P}=(X,\precsim)$ is called a {\it Frink ideal} in $X$ if $\delta (Z)\subseteq I$ for all finite
subsets $Z\subseteq I$.
In what follows, $Fid(X,\precsim)$
denote the set of all Frink ideals of $\mathcal{P}$.
We say that
$\mathcal{P}$ is {\it precontinuous} if and only if for each $x\in X$ there is a smallest $I\in Fid(X,\precsim)$
such that  $x\in I^{\uparrow\downarrow}$. }
\end{definition}

\begin{remark}\label{opa2}{\rm In a continuous lattice $\mathcal{L}$, $\turnnw{\twoheadrightarrow}{x}$ 
is automatically directed and thus we may write the axiom of approximation, which defines continuous lattices, as 
\begin{center}
$x=sup \turnnw{\twoheadrightarrow}{x}=sup\{u\in \mathcal{L}\vert u\ll x\}$ 
\end{center}
or as
\begin{center}
whenever $x\not\precsim y$, then there is a $u\ll x$ with $u\not\precsim y$.
\end{center}

}
\end{remark}

\begin{proposition}\label{vjh}{\rm Let $\mathcal{L}=(X,\precsim)$ be a continuous lattice. Then, 
$\mathfrak{X}=(X,\sigma(\mathcal{L}),$
$\omega(\mathcal{L}),\precsim)$ is a joincompact bitopological ordered space.}
\end{proposition}
\begin{proof} By \cite[Theorems III-1.9 and III-1.10]{GH}, $\sigma(\mathcal{L})\vee
\omega(\mathcal{L})=\lambda(\mathcal{L})$ is a compact topology.
To show that $\precsim$ is closed,
suppose that $x\not\precsim y$ for some $x, y\in X$. 
By the remark \ref{opa2} (see also \cite[Definition I-1.6]{GH}), there exists $z\in X$
such that $z\ll x$ and $z\not\precsim y$. Then, $X\setminus \uparrow{z}$ is a decreasing
$\omega$-open neighborhood of $y$ and 
$\turnnw{\twoheadleftarrow}{z}$
is an increasing Scott-open neighborhood of $x$ such that 
$\turnnw{\twoheadleftarrow}{z}\cap (X\setminus \uparrow{z})=\emptyset$.
It follows that $\precsim$ is $\sigma\times\omega$-closed subset of $X\times X$ and 
so $(X,\sigma,\omega,\precsim)$ is a bitopological ordered space. 
\end{proof}

\begin{remark}\label{opa1}{\rm
In $[0,1]$, $x\ll y$ if and only if $x<y$ or $x=y=0$. Thus this ordered set is a continuous lattice (see \cite[Example 3.2]{HKMS}). 
The {\it lower topology} on $[0,1]$ is $\mathfrak{L}=\{[0,a)\vert 0<a\precsim 1\}\cup \{\emptyset, [0,1]\}$ and the {\it upper topology} is 
$\mathfrak{U}=\{(a,1]\vert 0\precsim a<1\}\cup \{\emptyset, [0,1]\}$.
A function $f$ from a topological space $X$ into the real unit interval $[0,1]$ is Scott-continuous if and only if it is lower semicontinuous in the sense of classical analysis (see \cite[Example 1-5.9]{GW}). On the other hand, 
if $X$ is a topological space, then 
$f: X\to [0,1]$ is 
upper semicontinuous if and only if it is a continuous map of 
$f: X\to [0,1]$
with the lower topology (see \cite[Page 295]{HKMS}).
Hence, if
$(X,\sigma,\omega)=(X,\sigma(\mathcal{L}),\omega(\mathcal{L}))$ is the bitopological space generated by a continuous lattice
$\mathcal{L}=(X,\precsim)$, 
without loss of generality, by Proposition \ref{a2},
we may assume that
the notion of $\sigma$-lower  semicontinuous function
coincides with the notion of Scott-continuous function and the notion of
$\omega$-upper semicontinuous function
coincides with the notion of $\omega$-continuous function.
The following notation will be used in the sequel: 
$\mathfrak{I}=([0,1], \sigma, \omega)$.
}
\end{remark}

\begin{theorem}\label{kjh}{\rm Let $\mathcal{L}=(X,\precsim)$ be a continuous lattice. 
Then,  the set of all Scott and $\omega$-continuous functions is a 
Richter-Peleg 
multi-utility representation of $\precsim$.}
\end{theorem}
\begin{proof} Suppose that $\mathcal{L}=(X,\precsim)$ is a continuous lattice. 
By Proposition \ref{vjh},  $\mathfrak{X}=(X,\sigma(\mathcal{L}),\omega(\mathcal{L}),\precsim)$
is a joincompact bitopological ordered space.
Let $\mathcal{V}$ be the set of all 
Scott and $\omega$-continuous
Richter-Peleg utility representation of $\precsim$.
By Proposition \ref{a111} we have $\mathcal{V}\neq \emptyset$.
Therefore,  $\mathcal{V}$ is a Scott and $\omega$-continuous
Richter-Peleg 
multi-utility representation of $\precsim$.
\end{proof}

A function $p: X\to Y$ is a {\it quotient map} if satisfy the following conditions: 
(i) $p$ is surjective; (ii) $p$ is continuous (i.e. $U$ is open in $Y$ implies that
$p^{-1}(U)$ is open in $X$), and (iii) $V\subseteq Y$ and $p^{-1}(V)$ open in $X$ implies that $V$ open in $Y$.
In this case we say the map $f$ is a {\it quotient map}.
In fact, if $f$ is a quotient map, then,
\begin{center}
$U$ is open in $Y$ if and only if $p^{-1}(U)$ is open in $X$.
\end{center}
The topology produced by $p$ is known as {\it quotient topology}.
If $\approx$ is an equivalence relation on a topological space $(X,\tau)$,
then the quotient set by this equivalence relation $\approx$ will be denoted by 
$X_{_\approx}$, and its elements (equivalence classes) by $[x]$. Let the {\it projection map }
$\pi: X\to X_{_\approx}$
which carries each point of $X$ to the element of $X_{_\approx}$
that contains it. The projection map is a quotient map.
The quotient topology defined by $\pi$ is denoted by $\tau_{_\approx}$ and
the topological space $(X_{_\approx},\tau_{_\approx})$ is called the {\it quotient space}
of $X$ determined by $\approx$.
Thus, the typical open set in
$X_{_\approx}$
is a collection of equivalence classes whose union is an open set in $X$.
A continuous map $f: X\to Y$ {\it respects the equivalence relation} $R$ if equivalent points have identical
images, that is if $xRy$ implies $f(x)=f(y)$.

Let $\precsim$ be a preorder on a set $X$.
Define an equivalence class $\approx$ on $X$ by 
$x \approx y$ if and only if $x\precsim y$ and $y\precsim x$. Then, $[x]=\{y\in X\vert $
$x\precsim y$ and $y\precsim x$ $\}$.
We can define a partial order $\sqsubseteq$ on $X_{_\approx}$ by:
\begin{center}
$[x] \sqsubseteq [y]$ if and only if $[x]=[y]$ or there are $x^{\prime}\in [x], y^{\prime}\in [y]$ such that $x^{\prime}\precsim y^{\prime}$.
\end{center}
Clearly,
\begin{center}
$x\precsim y$ if and only if $[x]\sqsubseteq [y]$.
\end{center}

In the following, the symbols $X_{_\approx}$ and $\sqsubseteq$,
will always stand for the notions
that have
been defined just above.

\begin{remark}\label{olymb}{\rm Let $\mathcal{P}=(X,\precsim)$ be a preordered set and 
$\mathcal{Q}=(X_{_\approx},\sqsubseteq)$ be the poset, as it has been defined just
before. Let $Fid(X_{_\approx},\sqsubseteq)$ denote the set of all Frink ideals in 
$\mathcal{Q}$.
If $x, y\in X$, we write $[x] \ll^{\ast}_{_e} [y]$, if 
for each $[I]\in Fid(X_{_\approx},\sqsubseteq)$, $[y]\in [I]^{\delta}$ implies $[x]\in [I]$. 
}
\end{remark}

The following three propositions are useful in the proof of main theorem.

\begin{proposition}\label{tame}{\rm Let $\mathfrak{X}=(X,\tau_1,\tau_2,\precsim)$
be a topological preordered space and let
$\mathfrak{X}_{_\approx}=(X_{_\approx},\mathfrak{t}_{_1},\mathfrak{t}_{_2},\sqsubseteq)$ be the quotient space of $\mathfrak{X}$.
Then, $\mathfrak{X}$ is a topological preordered space if and only if 
$\mathfrak{X}$ is a topological ordered space.
}
\end{proposition}
\begin{proof} By Theorem \ref{a3}($\mathfrak{a}$) and by definition of the
quotient topology, it is easy to show that:\
$\precsim$ is $\tau_1\times \tau_2$-closed if and only if
$\sqsubseteq$ is
$\mathfrak{t}_{_1}\times \mathfrak{t}_{_2}$-closed. The rest is obvious.
\end{proof}

\begin{proposition}\label{tfte}{\rm Let $\precsim$ be a preorder on a set $X$.
Then, $(X,\precsim)$ is precontinuous if and only if $(X_{_\approx},\sqsubseteq)$ is precontinuous.
}
\end{proposition}
\begin{proof}
It is an immediate consequence of
Definition \ref{pas} and Remark \ref{olymb}.
\end{proof}

Scott's continuous functions are monotonic, so $x\approx y$ implies $f(x)=f(y)$ 
for a Scott continuous function $f$, that is, $f$ respects $\approx$.

The following proposition is an immediate consequence of \cite[Theorem 2.82]{mol}

\begin{proposition}\label{a1me}{\rm Let $\mathcal{P}=(X,\precsim)$ be a preordered set and let
$p$ be the quotient map of $X$ to $X_{_\approx}$.
Let $f: X\to [0,1]$
be a Scott and $\omega$-continuous Richter-Peleg representation of $\precsim$.
Then, $\sqsubseteq$ induces a 
Scott and $\omega$-continuous Richter-Peleg representation 
$\widetilde{f}: X_{_\approx}\to [0,1]$
such that $\widetilde{f}\circ \pi =f$.
Conversely, if $\widetilde{f}: X_{_\approx}\to [0,1]$ is 
a Scott and $\omega$-continuous Peleg-Richter representation 
of $\sqsubseteq$, then $\precsim $ induces a
Scott and $\omega$-continuous Peleg-Richter representation 
$f: X\to [0,1]$
such that $\widetilde{f}\circ \pi=f$.
}
\end{proposition}

The following corollary is an immediate consequence of Theorem \ref{kjh},
Remark \ref{olymb} and Propositions \ref{tame}, \ref{tfte} and \ref{a1me}.

\begin{corollary}\label{lant}{\rm Let $\mathfrak{X}=(X,\tau_1,\tau_2,\precsim)$
be a topological preordered space and let
$\mathfrak{X}_{_\approx}=(X_{_\approx},\mathfrak{t}_{_1},\mathfrak{t}_{_2},\sqsubseteq)$ be the quotient space of $\mathfrak{X}$.
Then, the set of 
Scott and $\omega$-continuous Peleg-Richter representation of $\precsim$
in
$\mathfrak{X}$ is isomorphic to the
set of 
Scott and $\omega$-continuous Peleg-Richter representation of $\sqsubseteq$
in
$\mathfrak{X}_{_\approx}$.
}
\end{corollary}

The following lemma is a result of Ern\'{e} in \cite{ern2}.

\begin{lemma}\label{aeks}{\rm Let $\mathcal{P}=(X,\precsim)$ be a poset. 
Then, the normal completion of $\mathcal{P}$ is a continuous lattice
if and only if $\mathcal{P}$
is precontinuous.
}
\end{lemma}

The following lemma is Theorem 1-3.12 in \cite{GW}.

\begin{lemma}\label{GW}{\rm A complete lattice $\mathcal{L}=(X,\precsim)$ is continuous if and only if there is an injection of $\mathcal{L}$ into some power $[0,1]^I$ of the unit interval preserving arbitrary meets and directed joins.}
\end{lemma}

\begin{definition}\label{cgs}{\rm 
Let $\{f_{_a}\vert a\in A\}$ be a collection of functions.
Then, the {\it evaluation map} $e: X \longrightarrow \displaystyle\prod_{a\in A}f_{_a}$
induced by the collection $\{f_{_a}\vert a\in A\}$
is defined as follows: for each $x\in X$, $e(x)=(f(x))_{_{a}}$.
That is, for each $x\in X$, $e(x)$ is the point in $\displaystyle\prod_{a\in A}f_{_a}$
whose ath coordinate is $f_{_a}(x)$ for each $a\in A$.}
\end{definition}

\begin{theorem}\label{gjh}{\rm Let $\mathcal{L}=(X,\precsim)$ be a preordered set.
Then,  $\precsim$ has the set of Scott and $\omega$-continuous
functions from $X$ into the real unit interval $[0,1]$ 
as a Richter-Peleg 
multi-utility representation if and only if $\precsim$ is precontinuous.}
\end{theorem}
\begin{proof} 
Let $X_{_\approx}$ be the quotient space of $X$. Since 
$(X,\precsim)$ is precontinuous,  by Proposition \ref{tfte}, we have that 
$(X_{_\approx},\sqsubseteq)$ is precontinuous.
Let
$(\delta(X_{_\approx}), \sqsubseteq^{\delta})$
be the MacNeille completion of $(X_{_\approx},\sqsubseteq)$. 
By Lemma \ref{aeks},
$(\delta(X_{_\approx}), \sqsubseteq^{\delta})$ is a continuous lattice. 
Let 
\begin{center}
$\mathcal{V}=\{g: \delta(X_{_\approx})\to [0,1]\vert $
$g$
is a
Scott and $\omega$-continuous Richter-Peleg representation of
$\sqsubseteq^{\delta}\}$.
\end{center}
By Theorem \ref{kjh},
we have that $\mathcal{V}$ is a Richter-Peleg multi-utility representation of
$\sqsubseteq^{\delta}$. Let $\varphi^{\delta}$ be the usual embedding of  
$(X_{_\approx},\sqsubseteq)$ to
$(\delta(X_{_\approx}), \sqsubseteq^{\delta})$.
Clearly, $\varphi^{\delta}$ preserves all existing joins and meets (see \cite[Theorem 1]{abi}).
 Then, for each $g\in \mathcal{V}$, $g\circ \varphi^{\delta}$
 is a Scott and 
 $\omega$-continuous Richter-Peleg function of $X_{_\approx}$ to $[0,1]$.
 Let
 \begin{center}
 $\widetilde{\mathcal{V}}=\{\widetilde{f}=g\circ \varphi^{\delta}\vert g\in \mathcal{V}\}$.
  \end{center}
It follows that $\widetilde{\mathcal{V}}$ is a Richter-Peleg multi-utility representation of
$\sqsubseteq$. 
Define,
 \begin{center}
 $\widehat{\mathcal{V}}=\{h=\widetilde{f}\circ \pi\vert \widetilde{f}\in
 \widetilde{\mathcal{V}}\}$, where $\pi$ is the projection map of $X$ to $X_{_\approx}$\}
 \end{center}
If $x, y\in X$, then
\begin{center}
$x\sim y\Rightarrow h(x)=h(y)$ and $x\prec y\Rightarrow h(x)<h(y)$.
\end{center}
 Then, by Proposition \ref{a1me}, we have that $h$
is a 
Scott and 
$\omega$-continuous function of $(X,\precsim)$ to $[0, 1]$.

Hence, $h$ is a Richter-Peleg utility representation of
$(X,\precsim)$.
It follows that
$\widehat{\mathcal{V}}$ is a Richter-Peleg multi-utility representation of
$(X,\precsim)$.

Conversely, suppose that $\mathcal{V}$ is a Richter-Peleg multi-utility representation of $\precsim$ such that each $f\in \mathcal{V}$ is a Scott and $\omega$-open
continuous function
from $X$ into the real unit interval $[0,1]$. 
For every $f\in \mathcal{V}$ we define the  preorder $\precsim_{_f}$ such that
\begin{center}
$\precsim_{_f}=\{(x,y)\in X\times X\vert f(x)\leq f(y)\}$.
\end{center}
Then,
the monotonicity and Scott and $\omega$-continuity
of $f$ implies that $\precsim_{_f}$ is a complete binary relation and
for every point $x\in X$, the set $\{y\in X\vert  y\precsim_{_f} x\}$ is $\sigma$-closed and
the set $\{z\in X\vert  x\precsim_{_f} z\}$
is $\omega$-closed in $X$.
It, thus, follows that $\precsim=\displaystyle\bigcap_{f\in \mathcal{V}}\precsim_{_f}$
is $\sigma\times\omega$-closed preorder and therefore $(X,\sigma,\omega,\precsim)$ 
is a bitopological ordered space.
As it is known, in $[0, 1]$, $x\ll y$, if and only if $x<y$ or $x=y=0$. Thus, $([0, 1],\ll)$ is a continuous lattice and thus
$\ll$ is the specialization order of its Scott topology.
Therefore,
$\mathbb{I}=([0,1],\mathfrak{L},\mathfrak{U})$, where $\mathfrak{L}, \mathfrak{U}$ are the lower and upper
topologies corresponds to the unit interval for bitopological ordered spaces (see Example \ref{ex0}).
For each $x\in X$, let
$h(x)=\{y\vert y\in X\ {\rm and}\ y\precsim x\}$. Then, the intersection of any $\{h(x)\vert x\in X\}$ is a lower cut in $X$
(see \cite[Definition 1]{abi}).
By \cite[Theorem 1]{abi}, $h(x)$ is a one -to- one mapping from $(X,\precsim)$
into $(\delta(X), \precsim^{\delta})$ which
extends $\precsim$ and preserves arbitrary meets and directed  joins.
 Define
\begin{center}
$\widehat{f}: \delta(X)\longrightarrow \mathbb{I}$ by $\widehat{f}(x)=\displaystyle\inf_{y\in x^{^\uparrow}}f(y)$.
\end{center}
Since $\precsim^{\delta}$ extends $\precsim$, by monotonicity of $f$ (it is Scott continuous), for each $y\in x^{^\uparrow}$
we have that $x\precsim y$. Therefore,
\begin{center}
$x \precsim^{\delta} y$ if and only if $\widehat{f}(x)\leq \widehat{f}(y)$.
\end{center}
The subset $[0,1]$ of reals is complete with respect to upper and lower topologies, thus by remark \ref{opa1}
we concludes that
$h$
is an
injection of $\delta(X)$ into $\mathbb{I}$ which preserves arbitrary meets and directed  joins.
On the other hand, if $x\in X$, then $\widehat{f}(x^{^\uparrow})=f(x)$. It follows that $\widehat{f}\circ h=f$.
For the sake of clarity, we'll assume that the above-mentioned family $\mathcal{V}$ satisfies the relation
$\mathcal{V}=\{f_{_a}\vert a\in A\}$
in the next section of the proof.
Let
\begin{center}
$\widehat{\mathcal{V}}=\{\widehat{f}_{_a}\in \delta(X)\vert \widehat{f}_{_a}\circ h=f_{_a}\}$. 
\end{center}
We correspond with
each $\widehat{f}_{_a}\in \widehat{\mathcal{V}}$ the set $\widehat{f}_{_a}(X)=I_{_{\widehat{f}_{_a}}} \subseteq [0,1]$ and denote
$\mathbb{I}_{_{f_{_a}}}=(I_{_{\widehat{f}_{_a}}},\mathfrak{L},\mathfrak{U})=(I_{_{\widehat{f}_{_a}}},\sigma,\omega)$ (remark \ref{opa1}).
Define 
the evaluation map 
\begin{center}
$e: \delta(X) \longrightarrow \displaystyle\prod_{\widehat{f}_{_a}\in \widehat{V}}\mathbb{I}_{_{f_{_a}}}$
\end{center}
induced by the collection $\{\widehat{f}_{_a}\vert a\in A\}$.
Then, for each $x\in X$, $e(x)$ is the point in 
$\displaystyle\prod_{\widehat{f}_{_a}\in \widehat{V}}\mathbb{I}_{_{f_{_a}}}$
whose ath coordinate is $\widehat{f}_{_a}(x)$ for each $a\in A$.
Since each $\widehat{f}_{_a}$ has a range contained in some closed and bounded interval $I_{_{\widehat{f}_{_a}}}$, by Theorem 8.16 in \cite{wil}, $\delta(X)$ is homeomorphic to a subspace of the cube $\displaystyle\prod_{\widehat{f}_{_a}\in \widetilde{V}}I_{_{\widehat{f}_{_a}}}$. Since $\precsim$ is $\sigma\times \omega$-closed, it separates points from closed sets and 
since the
Lawson topology $\lambda=\sigma\vee \omega$ in $\delta(X)$ is $T_1$, by Theorem 8.16 in \cite{wil}, we have that the evaluation map
$e: \delta(X) \longrightarrow \displaystyle\prod_{\widehat{f}_{_a}\in \widehat{V}}I_{_{\widehat{f}_{_a}}}$ is an injection.
On the other hand, since each $\widehat{f}_{_a}$ preserves arbitrary meets and directed  joins so does $e$.
By Lemma \ref{GW} we conclude that the normal completion $(\delta(X),\precsim^{\delta})$ is a continuous lattice.
 Therefore, by Lemma
\ref{aeks} we conclude that $(X,\precsim)$ is precontinuous.
\end{proof}

\par\bigskip\smallskip\par\noindent

\par\noindent
{\it Address}: {\tt {Athanasios Andrikopoulos} \\ {Department of Computer Engineering \& Informatics\\ University of Patras\\ Greece}
\par\noindent
{\it E-mail address}:{\tt aandriko@ceid.upatras.gr}

\end{document}